\documentclass[10pt,a4paper]{amsart}

\usepackage{amsmath}
\usepackage{amsthm}
\usepackage{enumerate} 
\usepackage{amssymb}
\usepackage{graphicx}
\usepackage[all]{xy}
\usepackage{hyperref}
\usepackage{aliascnt}
\usepackage{xcolor}
\usepackage{stmaryrd} 
\usepackage{mathtools}
\usepackage{verbatim} 
\usepackage{txfonts} 
\usepackage{lmodern}
\usepackage{setspace}   

\newtheorem{lma}{Lemma}[section]

\newaliascnt{thmCt}{lma}
\newtheorem{thm}[thmCt]{Theorem}
\aliascntresetthe{thmCt}

\newaliascnt{corCt}{lma}
\newtheorem{cor}[corCt]{Corollary}
\aliascntresetthe{corCt}

\newaliascnt{propCt}{lma}
\newtheorem{prop}[propCt]{Proposition}
\aliascntresetthe{propCt}

\newtheorem*{thm*}{Theorem}
\newtheorem*{cor*}{Corollary}
\newtheorem*{prop*}{Proposition}

\theoremstyle{definition}

\newaliascnt{prgCt}{lma}
\newtheorem{prg}[prgCt]{}
\aliascntresetthe{prgCt}

\newaliascnt{dfnCt}{lma}
\newtheorem{dfn}[dfnCt]{Definition}
\aliascntresetthe{dfnCt}

\newaliascnt{rmkCt}{lma}
\newtheorem{rmk}[rmkCt]{Remark}
\aliascntresetthe{rmkCt}

\newaliascnt{rmksCt}{lma}

\aliascntresetthe{rmksCt}

\newaliascnt{ntnCt}{lma}
\newtheorem{ntn}[ntnCt]{Notation}
\aliascntresetthe{ntnCt}

\newaliascnt{ntnsCt}{lma}

\aliascntresetthe{ntnsCt}

\newaliascnt{qstCt}{lma}
\newtheorem{qst}[qstCt]{Question}
\aliascntresetthe{qstCt}

\newaliascnt{prblCt}{lma}

\aliascntresetthe{prblCt}

\newaliascnt{exaCt}{lma}
\newtheorem{exa}[exaCt]{Example}
\aliascntresetthe{exaCt}

\newcommand{\C}{\mathbb{C}}
\newcommand{\N}{\mathbb{N}}

\newcommand{\cc}{\mathfrak{c}}
\newcommand{\dd}{\mathfrak{d}}

\DeclareMathOperator{\Th}{Th}

\DeclareMathOperator{\diam}{diam}

\newcommand{\CatCa}{\mathrm{C}^*}

\DeclareMathOperator{\Cu}{Cu}

\DeclareMathOperator{\CW}{CW}

\DeclareMathOperator{\AF}{AF}
\DeclareMathOperator{\Lsc}{Lsc}

\DeclareMathOperator{\Hom}{Hom}
\DeclareMathOperator{\id}{id}

\newcommand{\Cus}{\Cu^*}

\DeclareMathAlphabet{\mymathbb}{U}{bbold}{m}{n}

\begin{document}
\onehalfspacing
\title{Fra\"{i}ss\'{e} theory for Cuntz semigroups}

\author{Laurent Cantier}
\address{Laurent Cantier,\newline
Departament de Matem\`{a}tiques \\
Universitat Aut\`{o}noma de Barcelona \\
08193 Bellaterra, Spain\newline
Institute of Mathematics\\ Czech Academy of Sciences\\ Zitna 25\\ 115 67 Praha 1\\ Czechia}
\email[]{laurent.cantier@uab.cat}
\urladdr{www.laurentcantier.fr}

\author{Eduard Vilalta}
\address{Eduard Vilalta\newline
Fields Institute for Research in Mathematical Sciences\\
M5T 3J1 Toronto, Canada}
\email[]{evilalta@fields.utoronto.ca}
\urladdr{www.eduardvilalta.com}

\thanks{Both authors were partially supported by MINECO (grant No. PID2020-113047GB-I00), and by the Departament de Recerca i Universitats de la Generalitat de Catalunya (grant No. 2021-SGR-01015). The first author was also supported by the Spanish Ministry of Universities and the European Union-NextGenerationEU through a Margarita Salas grant. The second author was also supported by MINECO grant No. PRE2018-083419 and by the Fields Institute for Research in Mathematical Sciences.}

\keywords{Fra\"{i}ss\'{e} Theory, Cuntz semigroup, Cauchy sequences, $\Cu$-distance}

\begin{abstract} 
We introduce a Fra\"{i}ss\'{e} theory for abstract Cuntz semigroups akin to the theory of Fra\"{i}ss\'{e} categories developed by Kubi{\'s}. In particular, we show that any (Cuntz) Fra\"{i}ss\'{e} category has a unique Fra\"{i}ss\'{e} limit which is both universal and homogeneous. We also give several examples of such categories and compute their Fra\"{i}ss\'{e} limits. During our investigations, we develop a general theory of Cauchy sequences and intertwinings in the category $\Cu$.
\end{abstract}

\maketitle

\section{Introduction}

Fra\"{i}ss\'{e} Theory was introduced in \cite{R54} by Fra\"{i}ss\'{e} in the context of model theory with the intent of giving a generic method to construct countable homogeneous structures from their finitely-generated substructures. Since then, several adaptations of this method have appeared. These include, but are not limited to, projective Fra\"{i}ss\'{e} theory (\cite{IS06}), Fra\"{i}ss\'{e} theory for metric structures (\cite{Y15}) and, of late, Fra\"{i}ss\'{e} categories (\cite{GK20,K13}). The overall idea is to build a \textquoteleft large object\textquoteright, called the \emph{Fra\"{i}ss\'{e} limit}, which is unique, universal, and homogeneous with respect to a class of \textquoteleft small sub-objects\textquoteright, known as the \emph{Fra\"{i}ss\'{e} class}.
All these variations of Fra\"{i}ss\'{e} theory have produced interesting results in their corresponding areas. For example, well-known topological spaces such as the Cantor set, the pseudo-arc, the Lelek fan or the Menger curve have been expressed as (projective) Fra\"{i}ss\'{e} limits. (See \cite{BK22,BK15,IS06,K13}.) Similarly, as shown in \cite{K13}, the Gurari\u\i\ space is the Fra\"{i}ss\'{e} limit of the class of finite dimensional Banach spaces. In graph theory, the countable random graph also arises as the Fra\"{i}ss\'{e} limit of the class of finite graphs and, in model theory, Ehrenfeucht-Fra\"{i}ss\'{e}-like games can be built using Fra\"{i}ss\'{e} classes. 

Recently, Fra\"{i}ss\'{e} theory has been applied to the field of $\CatCa$-algebras (for example, in \cite{EFHKKM16,JV22,M17,Vig22}) and, in that setting, many well known $\CatCa$-algebras have been constructed as Fra\"{i}ss\'{e} limits. As proof, the Jiang-Su algebra $\mathcal{Z}$ ---introduced in \cite{JS99}, and which plays a central role in the classification of simple $\CatCa$-algebras--- can be seen as a Fra\"{i}ss\'{e} limit (\cite{M17}). Further, as noted in \cite{G21}, this construction can be used to (re)prove in simpler ways some of the properties of the algebra, such as its strong self-absorption. Among other examples, several stably projectionless $\CatCa$-algebras were also built as Fra\"{i}ss\'{e} limits in \cite{JV22}, and the existence of a universal AF-algebra was proved in \cite{GK20}.

In light of the recent bridge between Fra\"{i}ss\'{e} theory and $\CatCa$-algebras, it seems appropriate to explore Fra\"{i}ss\'{e} categories of (abstract) Cuntz semigroups. First introduced by Cuntz in \cite{C78}, the \emph{Cuntz semigroup} is a powerful invariant for $\CatCa$-algebras that codifies how positive elements are compared. In \cite{CEI08}, Coward, Elliott and Ivanescu introduced the category $\Cu$ of \emph{abstract Cuntz semigroups}, or \emph{$\Cu$-semigroups} for short. This rich subcategory of positively ordered monoids has been studied extensively (see, among many others, \cite{APRT21,APT20, APT20b, C21a, C21b, CRS10}) and this study has resulted in new results for $\CatCa$-algebras, such as the ones obtained in \cite{APRT22,C23b,C23a,ERS11,R12,TV22, TV23,T08,V23}. Further, this abstract category allows one to view the Cuntz semigroup as a continuous functor from the category of $\CatCa$-algebras to $\Cu$. The aim of this paper is to introduce a Fra\"{i}ss\'{e} theory in the category $\Cu$, while also giving examples and studying its relations to its $\CatCa$-algebraic counterpart.

The first obstruction that one finds when trying to mimic the past approaches is the general lack of a non-trivial,  enriched distance in the Hom-sets of $\Cu$. Although such a distance does exist for specific $\Cu$-semigroups (and has been exploited succesfully in a number of situations; see \cite{C23b,C23a,CE08,R12,RS10}), this approach is still too restrictive for our purposes. Instead, we will compare morphisms in $\Cu$ by using \emph{finite-set comparison}, an idea that had appeared implicity in the past (e.g. \cite{CE08,RS10}) but that was first given a name in \cite{C22}. 

Using this notion of comparison, we introduce \emph{Cauchy sequences} of morphisms and develop a general theory of \emph{one-} and \emph{two-sided intertwinings} in $\Cu$, which generalizes the results from \cite{C22}. We show that any Cauchy sequence has a unique limit, and that each intertwining induces a morphism with the expected properties; see Theorems \ref{thm:UniLimComp} and \ref{thm:2SidIntCu}. With these tools at our disposal, we are able to obtain the main result of the paper:

\begin{thm*}[{\ref{prp:FraisseLimCu}}]
Let $\cc$ be a Fra\"{i}ss\'{e} category. Then any Fra\"{i}ss\'{e} sequence $(S_i,\sigma_{i,j})_i$ has a $\overline{\cc}$-limit $(S,\sigma_{i,\infty})_i$ such that
\begin{itemize}
 \item[(i)] $S$ is unique up to isomorphism, that is, $S$ does not depend on the Fra\"{i}ss\'{e} sequence chosen.
 \item[(ii)] The set $\Hom_{\overline{\cc}}(D,S)$ is nonempty whenever $D$ is countably-based.
\end{itemize}

Further, assume that $\cc$ is contained in a category $\dd$ where every $\cc$-inductive sequence has a limit, and where every $\dd$-object is the limit of a $\cc$-sequence. If $\cc\subseteq\dd$ satisfies the almost factorization property, then
\begin{itemize}
 \item[(iii)] For any $C\in\cc$, any $\alpha,\beta\in \Hom_{\dd}(C,S)$ and any finite set $F\subseteq C$, there exists a $\dd$-isomorphism $\eta_F\colon S\overset{\cong}{\longrightarrow} S$ such that $\eta_F\circ\alpha \simeq_F\beta$.
\end{itemize}
\end{thm*}

In Part \autoref{subsec:ClassMorCat} of \autoref{sec:FraisseCu} we study the relations between this theorem and the Fra\"{i}ss\'{e} theory of $\CatCa$-algebras, while in \autoref{sec:examples} we provide a number of examples, listed below:
\begin{itemize}
 \item[(1)] Any dimensional $\Cu$-semigroup $S_p$ of infinite type is the Fra\"{i}ss\'{e} limit of the Fra\"{i}ss\'{e} category $\mathfrak{s}_{p}$.
 \item[(2)] There exists a universal dimensional $\Cu$-semigroup $\mathcal{S}$ which is the Fra\"{i}ss\'{e} limit of the Fra\"{i}ss\'{e} category $\mathfrak{s}_{\rm dim}$.
 \item[(3)] There exist countably many Fra\"{i}ss\'{e} categories $\mathfrak{e}_n$ whose Fra\"{i}ss\'{e} limits are simple, non-stably finite, not weakly purely infinite $\Cu$-semigroups.
 \item[(4)] The $\Cu$-semigroup $\Lsc(2^\N,\overline{\N})$, where $2^\N$ denotes the Cantor set, is the Fra\"{i}ss\'{e} limit of the Fra\"{i}ss\'{e} category $\mathcal{K}_{2^\N}$.
 \item[(5)] The $\Cu$-semigroup $\Lsc(\mathbb{P},\overline{\N})$, where $\mathbb{P}$ denotes the pseudo-arc, is the Fra\"{i}ss\'{e} limit of the Fra\"{i}ss\'{e} category $\mathcal{K}_{\mathbb{P}}$.
\item[(6)] The Cuntz semigroup of the Jiang-Su algebra is the the Fra\"{i}ss\'{e} limit of the Fra\"{i}ss\'{e} category $\mathcal{K}_Z$.
\end{itemize}

Some of the notions in this paper and in \cite{TV21} seem to hint at the right direction to develop a model theory of $\Cu$-semigroups, similar to the model theory of $\CatCa$-algebras from \cite{FHLRTVW21,FHS14}. We do not pursue this here, since this is an ellaborate task that will be done elsewhere.


\textbf{Organization of the paper.} \textbf{\autoref{sec:FraisseInter}} recalls the categorical Fra\"{i}ss\'{e} theory developed in \cite{K13}, where the reader can find the definition of a Fra\"{i}ss\'{e} category, a Fra\"{i}ss\'{e} sequence, and the fact that that any Fra\"{i}ss\'{e} category admits a unique Fra\"{i}ss\'{e} sequence, up to two-sided approximate intertwining.

We introduce \emph{(Cuntz) Fra\"{i}ss\'{e} categories} and their \emph{limit} in \textbf{\autoref{sec:FraisseCu}}. We start by recalling some preliminaries about the category $\Cu$ and the notion of \emph{finite-set comparison} for $\Cu$-morphisms (\autoref{dfn:CompMorph}). We show that Cauchy sequences with respect to finite-set comparison have a unique \emph{limit} (\autoref{thm:UniLimComp}), and we study approximate intertwinings in the category $\Cu$ (\autoref{thm:2SidIntCu}). Then, we define the \textquoteleft Cuntz analogue\textquoteright\ of a Fra\"{i}ss\'{e} category (\autoref{dfn:cufraissedfn}), a Fra\"{i}ss\'{e} sequence (\autoref{dfn:cufraisseseqdfn}), and the fact that that any Fra\"{i}ss\'{e} category admits a unique Fra\"{i}ss\'{e} sequence, up to two-sided approximate intertwining (\autoref{thm:existenceuniqueness}). Finally, a characterization of Fra\"{i}ss\'{e} limits is given (\autoref{prp:FraisseLimCu}). We finish the section by exploring the relations between Fra\"{i}ss\'{e} categories of $\CatCa$-algebras and Fra\"{i}ss\'{e} categories of Cuntz semigroups. (Part \autoref{subsec:ClassMorCat}.) 

\textbf{\autoref{sec:examples}} is divided in Parts \autoref{subsec:A}-\autoref{subsec:E}, which deal with the aforementioned examples. In \textbf{\autoref{sec:CuMetrics}} we define the \emph{Thomsen semigroup of a $\Cu$-semigroup} (\autoref{prg:thmsencu}) by using the generator $\mathbb{G}$ of the category $\Cu$. This allows us to define metrics on any $\rm Hom_{\Cu}$-set (\autoref{dfn:distCu}). We then explore the properties and several examples of such metrics (Examples \ref{exa:distTriv}-\ref{exa:UniBasis} and \autoref{rmk:LambdaFin}).


\textbf{Acknowledgments.} This research started when the second author visited the first author at the Czech Academy of Sciences. They are both grateful to the IMCAS for its hospitality and for providing a great working environment. The first author would also like to thank W. Kubi\'s for introducing him to Fra\"{i}ss\'{e} Theory.

\section{Preliminaries on Fra\"{i}ss\'{e} categories}
\label{sec:FraisseInter}
The aim of this section is to gather the main definitions and results about \emph{Fra\"{i}ss\'{e} categories}. These differ slightly, but include, the model-theoretical tools originally developed in \cite{R54}. The categorical approach described below has been developed by Kubi\'{s}, and we follow his notation and definitions from \cite{K13}. (See also \cite{BK22} and \cite{GK20}.) 

As mentioned in the introduction, the core idea of Fra\"{i}ss\'{e} theory is to produce \textquoteleft large\textquoteright\ objects that are universal and homogeneous in a generic sense ---these objects will be called \emph{Fra\"{i}ss\'{e} limits}--- for a given category of \textquoteleft small\textquoteright\ objects ---these categories will be called \emph{Fra\"{i}ss\'{e} categories}---. 

\begin{dfn}\label{dfn:MetEnr}
Let $\cc$ be a small category. We say that $\mathfrak{c}$ is \emph{metric-enriched} if 

\begin{itemize}
 \item[(i)] For any $A,B\in \cc$, the set $\Hom_\cc(A,B)$ is equipped with a metric $d_{(A,B)}$.
 \item[(ii)] For any $\alpha_1,\alpha_2\in \Hom_\cc(A,B)$ and $\beta\in \Hom_\cc(B,C)$, we have 
 \[
  d_{(A,C)}(\beta\circ\alpha_1,\beta\circ\alpha_2)\leq d_{(A,B)}(\alpha_1,\alpha_2).
 \]
 \item[(iii)] For any $\alpha\in \Hom_\cc(A,B)$ and $\beta_1,\beta_2\in \Hom_\cc(B,C)$, we have 
 \[
  d_{(A,C)}(\beta_1\circ\alpha,\beta_2\circ\alpha)\leq d_{(B,C)}(\beta_1,\beta_2).
 \]
\end{itemize}
Whenever the context is clear, we write $d_\cc$ instead of $d_{(A,B)}$.
\end{dfn}

\begin{dfn}\label{dfn:FraisseMetric}
Let $\cc$ be a metric-enriched category. We say that $\cc$
\begin{itemize}
 \item[(JEP)] satisfies the \emph{joint embedding property} if, for any $A_1,A_2\in \cc$, there exists $B\in \cc$ such that both $\Hom_\cc(A_1,B)$ and $\Hom_\cc(A_2,B)$ are nonempty.
 \item[(NAP)] satisfies the \emph{near amalgamation property} if, for any $\varepsilon>0$, and any $\cc$-morphisms $\alpha_1\in \Hom_\cc(A,B_1)$ and $\alpha_2\in \Hom_\cc(A,B_2)$, there exist $C\in \cc$ and $\cc$-morphisms $\beta_1\in \Hom_\cc(B_1,C)$ and $\beta_2\in \Hom_\cc(B_2,C)$ such that $d_\cc(\beta_1\circ\alpha_1,\beta_2\circ\alpha_2)<\varepsilon$.
 \item[(SEP)] is \emph{separable} if there exists a countable dominating subcategory $\mathfrak{s}\subseteq \cc$, that is,
 \begin{itemize}
  \item[$\bullet$] the set of $\mathfrak{s}$-morphisms is countable. (A fortiori, so is the set of $\mathfrak{s}$-objects.)
  \item[$\bullet$] for any $A\in \cc$ there exists $S\in \mathfrak{s}$ such that $\Hom_\cc(A,S)$ is nonempty.
  \item[$\bullet$] for any $\varepsilon>0$ and any $\cc$-morphism $\sigma\colon S\longrightarrow A$ with $S\in \mathfrak{s}$, there exist $T\in\mathfrak{s}$ and $\alpha\in\Hom_\mathfrak{c}(A,T)$ and $\nu\in\Hom_\mathfrak{s}(S,T)$ such that $d_\cc(\alpha\circ \sigma,\nu)<\varepsilon$.
 \end{itemize}
\end{itemize}

We say that $\cc$ is a \emph{Fra\"{i}ss\'{e} category} if $\cc$ satisfies (JEP), (NAP) and (SEP).
\end{dfn}

\begin{dfn}\label{dfn:CuFraisseSeq}
Let $\cc$ be a metric-enriched category. An inductive sequence $(F_i,\sigma_{i,j})_{i\in\N}$ is called a \emph{Fra\"{i}ss\'{e} sequence} if, for any $\varepsilon>0$ and any $\cc$-morphism $\gamma_i\colon F_i\longrightarrow C$, there exists 	a $\cc$-morphism $\gamma_j\colon C\longrightarrow F_j$ for some $j\geq i$ such that $d_\cc(\sigma_{i,j},\gamma_j \circ\gamma_i)<\varepsilon$.
\end{dfn}

\begin{thm}[cf. {\cite{K13}}]
\label{prp:FraisseCathasSeq}
Let $\cc$ be a Fra\"{i}ss\'{e} category. Then $\cc$ admits a Fra\"{i}ss\'{e} sequence which is unique up to two-sided approximate intertwining. 
\end{thm}


Let us now recall the notion of \emph{Fra\"{i}ss\'{e} limits}. As stated before, these objects are often \textquoteleft large\textquoteright, in the sense that they do not belong to the Fra\"{i}ss\'{e} category $\cc$ at hand. Instead, they are built as inductive limits of  $\cc$-objects. (Particularly, as inductive limits of Fra\"{i}ss\'{e} sequences.)

More concretely, this amounts to the fact that a Fra\"{i}ss\'{e} category $\cc$ need not have inductive limits. Because of this, one often considers an inclusion of categories of the form $\cc\subseteq\dd$, where $\dd$ does have inductive limits. However, the category $\dd$ cannot be \textquoteleft too\textquoteright\ large, since this inclusion is asked to satisfy the following \emph{almost factorization property}.

\begin{dfn}\label{dfn:KubisAFP}
Let $\cc,\dd$ be metric-enriched categories such that $\dd$ has inductive limits and $\cc\subseteq \dd$. We say that the inclusion $\cc\subseteq \dd$ has the \emph{almost factorization property} if, for any inductive system $(A_i,\sigma_{i,j})_{i\in\N}$ in $\cc$ with $\dd$-limit $(A,\sigma_{i,\infty})_{i}$, any $B\in \cc$, any $\dd$-morphism $\beta\colon B\longrightarrow A$, and any $\varepsilon >0$, there exist $i\in \mathbb{N}$ and a $\cc$-morphism $\beta_\varepsilon\colon B\longrightarrow A_i$ such that $d_\dd(\sigma_{i,\infty}\circ \beta_{\varepsilon},\beta)<\varepsilon$.
\end{dfn}

\begin{thm}[cf. {\cite{K13}}]
\label{thm:fraisselim}
Let $\cc$ be a Fra\"{i}ss\'{e} category included in a category $\dd$ which admits inductive limits, and such that any object in $\dd$ is a limit of a $\cc$-sequence. Then, any Fra\"{i}ss\'{e} sequence in $\cc$ has a $\dd$-limit $F$ satisfying the following properties:
\begin{itemize}
 \item[(i)] $F$ is unique up to isomorphism.
 \item[(ii)] For any $D\in \dd$, the set $\Hom_\dd(D,F)$ is nonempty.
 \end{itemize}
If, additionally, $\cc\subseteq \dd$ satisfies the almost factorization property, then 
\begin{itemize}
 \item[(iii)] For any $\varepsilon>0$, any $C\in \cc$, and any $\alpha_1,\alpha_2\in \Hom_\dd(C,F)$, there exists a $\dd$-isomorphism $\eta\colon F\overset{\cong}{\longrightarrow} F$ such that $d_\dd(\eta \circ\alpha_1,\alpha_2)<\varepsilon$.
\end{itemize}
\end{thm}

\begin{rmk}\label{rmk:FraisseAUE}
The category of $\CatCa$-algebras $\CatCa$ is metric-enriched by the usual norm-distance between $^*$-homomorphisms. Therefore, the definitions given here can be applied to $\CatCa$ directly (as done in \cite{GK20}). However, it is more common to compare $^*$-homomorphisms metrically on finite sets. (See e.g. \cite{Vig22}.) For example, (NAP) in \autoref{dfn:FraisseMetric} gets changed to: for any $\varepsilon$, $\alpha_1,\alpha_2$, and finite set $F\subseteq A$, there exist $^*$-homomorphism $\beta_1,\beta_2$ such that $\Vert\beta_1\circ\alpha_1 (x)-\beta_2\circ\alpha_2 (x)\Vert<\varepsilon$ for every $x\in F$. A similar change is done in the third condition of (SEP) and the definition of Fra\"{i}ss\'{e} sequence.

As we will discuss in \autoref{rmk:endingpartc}, for our purposes it would also be interesting to study Fra\"{i}ss\'{e} categories of $\CatCa$-algebras with respect to approximate unitary equivalence. In that version, (NAP) would be changed to: for any $\varepsilon$, $\alpha_1,\alpha_2$, and finite set $F\subseteq A$, there exist $^*$-homomorphism $\beta_1,\beta_2$ and a unitary $u\in\tilde{C}$ such that $\Vert u^*\beta_1\circ\alpha_1 (x)u-\beta_2\circ\alpha_2 (x)\Vert<\varepsilon$ for every $x\in F$. Analogous changes would be made to the other definitions.
\end{rmk}

\section{\texorpdfstring{Fra\"{i}ss\'{e} categories of Cuntz semigroups}{Fraisse categories of Cuntz semigroups}}\label{sec:FraisseCu}

As explained in the previous section, the approach to Fra\"{i}ss\'{e} categories from \cite{K13} requires each ${\rm Hom}$-set to be equipped with a right- and left-enriched metric. This rarely happens in the category $\Cu$: Every ${\rm Hom}_{\Cu}$-set admits a natural metric (defined and studied in the last section of this paper), but such a metric is seldom left-enriched. This was already the case for the specific instances of this metric considered in the past. (See e.g. \cite{C22} and \cite{CE08}.)

To overcome these constraints, we introduce a theory of Fra\"{i}ss\'{e} categories for $\Cu$-semigroups where, instead of using a metric on the ${\rm Hom}$-sets, we compare morphisms on finite sets. This allows us to bypass any sort of enrichment property. In the course of our investigations, we also define an analog of Cauchy sequences for $\Cu$-morphisms, which are shown to have a unique limit. Further, building on the results from \cite{C22}, we develop a general theory of one- and two-sided approximate intertwinings in $\Cu$. These tools allow us to define and obtain analogous notions and results to those of Kubi\'{s}. We finish the section by describing the relation between these theories in the context of $\CatCa$-algebras and concrete Cuntz semigroups.

First, let us recall some preliminaries about the category $\Cu$.

\begin{prg}[\textbf{$\Cu$-semigroups}]\label{pgr:Cuntz}
Let $x,y$ be elements in a partially ordered set $P$. We write $x\ll y$ if, for every increasing sequence $(z_n)_n$ which has a supremum such that $y\leq \sup_n z_n$, there exists $n\in\N$ such that $x\leq z_n$.

 As defined in \cite{CEI08}, a positively ordered monoid $S$ is said to be a \emph{$\Cu$-semigroup} if $S$ satisfies the following properties:
 \begin{itemize}
  \item[(O1)] Every increasing sequence in $S$ has a supremum.
  \item[(O2)] Every element in $S$ can be written as the supremum of a $\ll$-increasing sequence.
  \item[(O3)] The addition and the $\ll$-relation are compatible.
  \item[(O4)] Suprema of increasing sequences and the addition are compatible.
 \end{itemize}

 A monoid morphism between $\Cu$-semigroups is a \emph{$\Cu$-morphism} if it preserves the order, the $\ll$-relation, and suprema of increasing sequences. We denote the category of $\Cu$-semigroups and $\Cu$-morphisms by $\Cu$. (See e.g. \cite{APT18} or \cite{GP22} for a more detailed exposition.)

The \emph{Cuntz semigroup} of a $\CatCa$-algebra $A$, denoted by $\Cu(A)$, is the quotient $(A\otimes\mathcal{K})_+/\!\sim$ equipped with the addition induced by diagonal addition and the order induced by $\precsim$, where the relations $\precsim$ and $\sim$ are defined as follows:
 \[
\begin{split}
  a\precsim b &:\!\iff a=\lim_n r_n br_n^*\text{ for some sequence }(r_n)_n\subseteq A\otimes \mathcal{K}.\\
  a\sim b &:\!\iff  a\precsim b\text{ and }b\precsim a.
\end{split}
 \]

 The Cuntz semigroup of $A$, first considered in \cite{C78}, is always a $\Cu$-semigroup. (See \cite{CEI08}.) Further, every $^*$-homomorphism from $A$ to $B$ induces a $\Cu$-morphism from $\Cu (A)$ to $\Cu (B)$.

 The relation between concrete Cuntz semigroups and the abstract category $\Cu$ has been studied extensively. (See \cite{APRT21, APRT22, APT20, C21a, C21b, CRS10,TV22, TV23} among many others.) For instance, it is known that the category $\Cu$ has direct limits and that the functor $\Cu$ is continuous. (See \cite{APT18}, \cite{CEI08}.)

 A $\Cu$-semigroup $S$ is \emph{countably-based} if $S$ contains a countable, sup-dense subset. Examples include the Cuntz semigroup of any separable $\CatCa$-algebra.
\end{prg}

\subsection{\texorpdfstring{Comparison of $\Cu$-morphisms}{Comparison of Cu-morphisms}}

As mentioned at the beginning of this section, it is not clear when a set of $\Cu$-morphisms can be equipped with a (meaningful) enriching metric. In order to overcome this issue and work in the general setting, we will compare $\Cu$-morphisms on finite sets. This notion was  introduced explicitly in \cite{C22}, although the idea had also appeared implicitly in the past when working with specific families of $\Cu$-morphisms. (See e.g. \cite{C23b,C23a,CE08,L22,R12}.)

\begin{dfn}[{\cite[Definition 3.9]{C22}}]\label{dfn:CompMorph}
Given a pair of $\Cu$-morphisms $\alpha,\beta\colon S\longrightarrow T$ and a finite subset $F\subseteq S$, we say that \emph{$\alpha$ and $\beta$ compare on $F$}, and we write $\alpha\simeq_F \beta$, if for any pair $x',x\in F$ with $x'\ll x$, we have
\[
  \alpha (x')\leq \beta (x)
  \quad\text{and}\quad
  \beta (x')\leq \alpha (x).
 \]
\end{dfn}

\begin{rmk}\label{rmk:CharacEqFiSetComp}
As observed in \cite{C22}, for any two $\Cu$-morphisms $\alpha,\beta\in\rm Hom_{\Cu}(S,T)$ the following are equivalent:
\begin{itemize}
\item[(i)] $\alpha=\beta$.
\item[(ii)] $\alpha\simeq_F \beta$ for any finite subset $F\subseteq S$.
\item[(iii)] $\alpha\simeq_{\{s',s\}} \beta$ for any $s',s\in S$ with $s'\ll s$.
\end{itemize}

Note that finite-set comparison can also be used for weaker forms of morphisms between $\Cu$-semigroups. (See \cite[Definition 2.2]{C23b}.)
\end{rmk}

With this notion of comparison at hand, we can define Cauchy sequences, limits, and approximate intertwinings in the category $\Cu$.

\begin{dfn}\label{dfn:CauchySeq}
Let $(\alpha_i)_i$ be a sequence of $\Cu$-morphisms in $\Hom_{\Cu}(S,T)$. We say that $(\alpha_i)_i$ is a \emph{Cauchy sequence} if, for any finite subset $F\subseteq S$, there exists an index $i_F$ such that  $\alpha_j\simeq_F \alpha_k$ whenever $j,k\geq i_F$.
\end{dfn}

The following definition of convergence was introduced in \cite[Definition 5.1]{L22} for concrete Cuntz morphisms. We give here the definition for any sequence of $\Cu$-morphisms.

\begin{dfn}\label{dfn:LimitLin}
Let $(\alpha_i)_i$ be a sequence of $\Cu$-morphisms in $\Hom_{\Cu}(S,T)$. We say that $(\alpha_i)_i$ \emph{converges towards a $\Cu$-morphism $\alpha\colon S\longrightarrow T$} if, for any finite subset $F\subseteq S$, there exists an index $i_F$ such that $\alpha_j\simeq_F \alpha$ whenever $j\geq i_F$.
\end{dfn}

\begin{rmk}
In view of our previous remark, note that $(\alpha_i)_i$ converges to $\alpha$ if, for any pair of elements $x',x\in S$ with $x'\ll x$, there exists an index $i_0$ such that $\alpha_j(x')\leq \alpha(x)$ and $\alpha(x')\leq \alpha_j(x)$ whenever $j\geq i_0$.
\end{rmk}

\begin{exa}\label{exa:ConvCa}
If a sequence of $^*$-homomorphisms $(\varphi_i\colon A\longrightarrow B)_i$ converges in norm to a morphism $\varphi\colon A\longrightarrow B$, then the sequence $(\Cu (\varphi_i))_i$ converges to $\Cu (\varphi)$. 

Indeed, given $[a],[b]\in \Cu (A)$ with $[a]\ll [b]$, we can find $\varepsilon >0$ small enough such that $[a]\leq [(b-\varepsilon )_+]$. Moreover, we can find $i\in \N$ big enough such that $\Vert \varphi_j - \varphi\Vert< \varepsilon$ for every $j\geq i$. Therefore, we have
$\varphi_j (a)\precsim \varphi_j ( (b-\varepsilon)_+ )=(\varphi_j (b)-\varepsilon )_+\precsim \varphi (b)$ and $ 
\varphi (a)\precsim\varphi ( (b-\varepsilon)_+ ) = (\varphi (b)-\varepsilon )_+\precsim \varphi_j (b)$,
as desired.
\end{exa}

\begin{thm}\label{thm:UniLimComp}
Let $S,T$ be $\Cu$-semigroups. Then any Cauchy sequence $(\alpha_i)_i$ in ${\rm Hom}_{\Cu}(S,T)$ converges towards a unique $\Cu$-morphism.
\end{thm}
\begin{proof}
Let us first assume that $S$ is countably-based, so that there exists a $\subseteq$-increasing sequence $(B_n)_n$ of finite sets of $S$ such that $B:=\bigcup\limits_{n\in\N}B_n$ is sup-dense in $S$. Let $(\alpha_i)_i$ be a Cauchy sequence in ${\rm Hom}_{\Cu}(S,T)$. We can find a strictly increasing map $\varphi\colon\N\longrightarrow \N$ such that $\alpha_j\simeq_{B_{n}}\alpha_k$ for any $j,k\geq \varphi(n)$.

Let ${\rm Seq}_{\ll}(B)$ denote the set of $\ll$-increasing sequence in $B$, which we may think of as maps $f\colon\N\longrightarrow B$, $i\mapsto f_i$, such that $f_i\ll f_{i+1}$ in $S$ for each $i$. With this notation, there exists a map $\psi \colon\N\times {\rm Seq}_{\ll}(B)\longrightarrow \N$ such that
\begin{itemize}
\item[(i)] $\psi (\cdot ,f)\colon\N\longrightarrow\N$ is strictly increasing for every fixed $f$.
\item[(ii)] $\{f_0,\ldots,f_{l+1}\}\subseteq B_{\psi (l,f)}$ for each $l\in\N$.
\end{itemize}
Note that the map $\varphi\circ\psi (\cdot ,f)$ is strictly increasing for any $f\in {\rm Seq}_{\ll}(B)$. 

Now fix $l\in\N$. By the definition of $\varphi$ we have that $\alpha_j\simeq_{B_{\psi(l,f)}} \alpha_k$ for any $j,k\geq \varphi(\psi(l,f))$. Further, it follows from the construction of $\phi$ that
\[
 \alpha_j(f_i)\leq \alpha_k(f_{i+1}) \quad\text{and}\quad \alpha_k(f_{i})\leq \alpha_j(f_{i+1})
\]
for any $0 \leq i\leq l$ and any $j,k\geq \varphi(\psi(l,f))$.

In particular, for $j=\varphi(\psi(l,f))$, $k=\varphi(\psi(l+1,f))$ and $i=l$, we obtain
\[
\alpha_{\varphi(\psi(l,f))} (f_l)\leq 
\alpha_{\varphi(\psi(l+1,f))} (f_{l+1}).
\]
In other words, the sequence $\left( \alpha_{\varphi(\psi(l,f))}(f_l) \right)_l$ is increasing. Consequently, we can define the following map
\[
\begin{array}{ll}
\alpha_{\rm Seq}\colon {\rm Seq}_{\ll}(B)\longrightarrow T\\
\hspace{2,1cm} f\longmapsto \sup_l  \alpha_{\varphi(\psi(l,f))}(f_l)
\end{array}
\]

We aim to construct a $\Cu$-morphism $\alpha\colon S\longrightarrow T$ induced by $\alpha_{\rm Seq}$. For this, we will need the following claim.

\textbf{Claim.} Let $f,g\in {\rm Seq}_{\ll}(B)$ be such that $\sup f\leq \sup g$. Then $\alpha_{\rm Seq}(f)\leq \alpha_{\rm Seq}(g)$.

\textit{Proof of the Claim.} Let $f,g$ be as in the statement. For any $l\in\N$, there exists $m\in\N$ such that $f_l\ll f_{l+1} \ll g_m$. Since $\varphi\circ\psi (\cdot ,g)$ is strictly increasing, we can assume that $m$ is large enough so that $\varphi(\psi(l,f))\leq \varphi(\psi(m,g))$. By construction, we have $\alpha_{\varphi(\psi (l,f))}(f_l)\leq \alpha_{k}(f_{l+1})$ for any $k\geq \varphi(\psi(l,f))$. Thus, we compute
\[
\alpha_{\varphi(\psi(l,f))}(f_l)
\leq \alpha_{\varphi(\psi(m,g))}(f_{l+1}) 
\leq \alpha_{\varphi(\psi(m,g))}(g_m)\leq \alpha_{\rm Seq} (g)
\]
which implies that $\alpha_{\rm Seq}(f)\leq \alpha_{\rm Seq}(g)$ and proves the claim.

Since $B$ is dense in $S$, we are now able to construct the following order-preserving map
\[
\begin{array}{ll}
\alpha\colon S\longrightarrow T\\
\hspace{0,55cm} s\longmapsto \alpha_{\rm Seq}((s_i)_i)
\end{array}
\]
where $(s_i)_i$ is any $\ll$-increasing sequence in $B$ obtained from (O2), whose supremum is $s$. (The claim shows that $\alpha$ is well-defined, i.e. $\alpha$ does not depend on the sequence $(s_i)_i$, and also that $\alpha$ preserves the order.) 

Further, using (O4), it is readily checked that $\alpha$ preserves the addition. Using a diagonal-type argument (see e.g. the proof of \cite[Lemma 3.12]{C22}), it can also be shown that $\alpha$ preserves suprema of increasing sequences. 

We are left to show that $\alpha$ preserves the compact-containment relation. Let $f,g\in {\rm Seq}_{\ll}(B)$ be such that $\sup f\ll \sup g$. Then, there exists $m\in\N$ such that $f_l\ll  g_{m-2} \ll g_{m-1}\ll g_m$ for any $l\in\N$. Find $l_0$ big enough such that $\varphi(\psi(m,g))\leq \varphi(\psi(l,f)) $ for any $l\geq l_0$. By construction, we have 
\[
  \alpha_{\varphi(\psi(l,f))}(f_l)\ll \alpha_{\varphi(\psi(l,f))}(g_{m-2}) \leq \alpha_{\varphi(\psi(m,g))}(g_{m-1})\ll \alpha_{\varphi(\psi(m,g))}(g_{m})\leq 
  \alpha_{\rm Seq}(g)
\]
whenever $l\geq l_0$. In particular, we get $\alpha_{\rm Seq}(f)\ll \alpha_{\rm Seq}(g)$ and, hence, $\alpha$ preserves the $\ll$-relation. This shows that $\alpha$ is a well-defined $\Cu$-morphism.

Finally, let us prove that $(\alpha_i)_i$ converges to $\alpha$. Let $x',x\in S$ be such that $x'\ll x$. By density of $B$ in $S$, there exists $f\in {\rm Seq}_{\ll}(B)$ such that $x' \ll \sup f \ll x$. 
On the one hand, note that there exists $l\in\N$ big enough such that $x'\ll f_l\ll f_{l+1}\ll x$. We deduce that, for any $i\geq  \varphi(\psi(l+1,f))$,  we have
 \[
 \alpha_i(x')\leq \alpha_i(f_l) \leq \alpha_{\varphi(\psi(l+1,f))}(f_{l+1})\leq \alpha_{\rm Seq}(f)= \alpha(\sup f) \leq \alpha(x).
 \]
On the other hand, note that $\alpha (x')\ll \alpha_{\rm Seq}(f)$ and, hence, there exists $m\in\N$ big enough such that $\alpha(x')\ll \alpha_{\varphi(\psi(m,f))}(f_m)$. We deduce that, for any $i\geq \varphi(\psi(m,f))$, we have
\[
 \alpha (x')\leq \alpha_{\varphi(\psi(m,f))}(f_m)\leq \alpha_i (f_{m+1}) \leq \alpha_i (x).
\]
This shows that $(\alpha_i)_i$ converges towards $\alpha$. Since such an $\alpha$ is unique, this also proves that $\alpha$ does not depend on the basis $B$ chosen.\\

Now let us show that the result holds for any (possibly not countably-based) $\Cu$-semigroup $S$.
Let $(\alpha_i)_i$ be a Cauchy sequence in ${\rm Hom}_{\Cu}(S,T)$. Any countably-based sub-$\Cu$-semigroup $H$ of $S$ (i.e. $H$ is $\Cu$-semigroup that order-embeds into $S$) naturally induces a Cauchy sequence $({\alpha_i}_{\mid H})_i$ by restriction. We can thus construct its limit as above, which we denote by $\alpha_{H}$.

Let $x\in S$. It follows from \cite[Lemma 5.1]{TV21} that there exists a (possibly not unique) countably-based sub-$\Cu$-semigroup $H_x$ such that $x\in H_x$. Let $H_1,H_2$ be countably-based sub-$\Cu$-semigroups that contain $x$. By construction, there exist $\ll$-increasing sequences $(x_n)_n$ and $(x_n')_n$ in $H_1$ and $H_2$ respectively with supremum $x$ such that $\alpha_{H_1}(x)=\sup_n\alpha_n (x_n)$ and $\alpha_{H_2}(x)=\sup_n\alpha_n (x_n')$. Further, we can choose these sequences so that, for any $m\geq n$, we have
\[
 \alpha_n (x_n)\leq \alpha_m (x_{n+1})\quad\text{and}\quad 
 \alpha_n (x_n')\leq \alpha_m (x_{n+1}')
 .
\]
Let $n\in\N$, and find $m\geq n$ such that $x_{n+1}\ll x_m'$. Then, we get 
\[
 \alpha_n (x_n)\leq \alpha_m (x_{n+1})\leq \alpha_m (x_m')\leq \alpha_{H_2}(x).
\]
Taking supremum over $n$ we obtain $\alpha_{H_1}(x)\leq \alpha_{H_2}(x)$ and, by a symmetric argument, we also get $\alpha_{H_2}(x)\leq \alpha_{H_1}(x)$. We conclude that $\alpha_{H_1}(x)= \alpha_{H_2}(x)$ for any $x\in S$ and, consequently, that the following map is well-defined
\[
\begin{array}{ll}
\alpha\colon S\longrightarrow T\\
\hspace{0,6cm} x\longmapsto \alpha_{H_x}(x)
\end{array}
\]

Using the techniques from \cite{TV21} one can check that $\alpha$ is a $\Cu$-morphism and that the sequence $(\alpha_i)_i$ converges to $\alpha$ (by construction). This ends the proof.
\end{proof}

\begin{ntn}
We have just shown that any Cauchy sequence $(\alpha_i)_i$ in $\Hom_{\Cu}(S,T)$ converges towards a unique $\Cu$-morphism $\alpha\colon S\longrightarrow T$. We will say that $\alpha$ is the \emph{limit} of the sequence, and write $\lim_i \alpha_i=\alpha$. 
\end{ntn}

We are now able to define and study approximate intertwinings in the category $\Cu$. This generalizes the concepts introduced in \cite{C22} for the specific case of uniformly-based $\Cu$-semigroups. 

\begin{dfn}\label{dfn:OneSidInt}
Let $(S_i,\sigma_{i,j})_{i\in\N}$ and $(T_i,\tau_{i,j})_{i\in\N}$ be two inductive sequences in $\Cu$. Assume that there exists a strictly  increasing map $\varphi\colon\N\longrightarrow \N$ together with a sequence of $\Cu$-morphisms  $(\alpha_i\colon S_i\longrightarrow T_{\varphi(i)})_i$ satisfying the following property:

For any finite subset $F\subseteq S_i$, there exists an index $i_F\geq i$ such that, for any $j\geq i_F$ and any $k\geq j$, the diagram
 \[
   \xymatrix{
     S_i\ar[r]^{\sigma_{i,j}} & S_{j}\ar[d]_{\alpha_{j }}\ar[rr]^{\sigma_{j,k }} && S_k\ar[d]^{\alpha_k} \\
   & T_{\varphi(j)}\ar[rr]_{\tau_{\varphi(j),\varphi(k) }} && T_{\varphi(k)}
   } 
   \]
approximately commutes within $F$, that is, $\alpha_k\circ\sigma_{i,k}\simeq_F \tau_{\varphi(j),\varphi(k)}\circ\alpha_j\circ\sigma_{i,j}$.

We say that $(\alpha_i)_i$ is a \emph{one-sided approximate intertwining}.
\end{dfn}

\begin{prop}\label{prop:induced}
Let $(S_i,\sigma_{i,j})_{i\in\N}$ and $(T_i,\tau_{i,j})_{i\in\N}$ be two inductive sequences in $\Cu$ and let $(\alpha_i\colon S_i\longrightarrow T_{\varphi(i)})_i$ be a one-sided approximate intertwining. Then there exists a $\Cu$-morphism $\alpha\colon S\longrightarrow T$ such that, for any finite subset $F\subseteq S_i$, there exists $i_F\geq i$ such that, for any $j\geq i_F$, the diagram
   \[
   \xymatrix{
     S_i\ar[r]^{\sigma_{i,j}} & S_{j}\ar[d]_{\alpha_{j }}\ar[rr]^{\sigma_{j,\infty }} && S\ar@{.>}[d]^{\exists \alpha} \\
   & T_{\varphi(j)}\ar[rr]_{\tau_{\varphi(j),\infty }} && T
   } 
   \]
approximately commutes within $F$, that is, $\alpha\circ\sigma_{i,\infty}\simeq_F \tau_{\varphi(j),\infty}\circ\alpha_{j}\circ\sigma_{i,j}$. 
\end{prop}

\begin{proof}
For any $i\in\N$ the sequence $(\tau_{\varphi(j),\infty}\circ\alpha_j\circ\sigma_{i,j})_j$ is Cauchy. It follows from \autoref{thm:UniLimComp} that the sequence has a limit, which we denote by $\eta_i\colon S_i\longrightarrow T$. Using that this limit is unique, it is readily checked that $\eta_i=\eta_j\circ\sigma_{i,j}$ for any $i\leq j$. By the universal property of direct limits, this induces a $\Cu$-morphism $\alpha\colon S\longrightarrow T$ such that $\eta_i=\alpha\circ \sigma_{i,\infty}$ for any $i\in\N$. Finally, note that for any finite subset $F\subseteq S_i$, we can find an index $i_F$ such that $\eta_i\simeq_F \tau_{\varphi(j),\infty}\circ\alpha_{j}\circ\sigma_{i,j}$ for any $j\geq i_F$. Thus, we have
\[
\alpha\circ \sigma_{i,\infty}
= \eta_i 
\simeq_F \tau_{\varphi(j),\infty}\circ\alpha_{j}\circ\sigma_{i,j}
\]
which ends the proof.
\end{proof}

\begin{exa}\label{exa:IntCaCu}
Let $(\psi_i\colon A_i\longrightarrow B_{\varphi(i)})_i$ be a one-sided approximate intertwining of $\CatCa$-algebras from $(A_i,\varphi_{i,j})_{i\in\N}$ to $(B_i,\phi_{i,j})_{i\in\N}$. It follows from \autoref{exa:ConvCa} that the induced maps $(\Cu (\psi_i))_i$ define a one-sided approximate intertwining from $(\Cu (A_i),\Cu(\varphi_{i,j}))_{i}$ to $(\Cu(B_i),\Cu(\phi_{i,j}))_{i}$.
\end{exa}

\begin{dfn}
\label{dfn:atwosided}
Let $(S_i,\sigma_{i,j})_{i\in\N}$ and $(T_i,\tau_{i,j})_{i\in\N}$ be two inductive sequences in $\Cu$. 
Assume that there exist two strictly increasing maps $\varphi,\psi\colon\N\longrightarrow \N$ together with two sequences of $\Cu$-morphisms $(\alpha_i\colon S_i\longrightarrow T_{\varphi(i)})_i$ and $(\beta_i\colon T_i\longrightarrow S_{\psi(i)})_i$ satisfying the following property:

For any finite sets $F\subseteq S_i$ and $G\subseteq T_i$, there exist indices $i_{F},i_G\geq i$ such that, for any $j\geq i_F$, $j'\geq i_G$ and any $k\geq \varphi(j)$, $k'\geq \psi(j')$, the diagrams
\[
   \xymatrix{
     S_i\ar[r]^{\sigma_{i,j}} & S_{j}\ar[d]_{\alpha_{j }}\ar[rr]^{\sigma_{j,\psi(k) }} && S_{\psi(k)} \\
   & T_{\varphi(j)}\ar[rr]_{\tau_{\varphi(j),k }} && T_{k}\ar[u]_{\beta_k}
   } 
  \quad \quad \quad \quad 
  \xymatrix{
     & S_{\psi(j')}\ar[rr]^{\sigma_{\psi(j'),k' }} && S_{k'}\ar[d]^{\alpha_{k'}} \\
 T_i\ar[r]_{\tau_{i,j'}}   & T_{j'}\ar[u]^{\beta_{j' }}\ar[rr]_{\tau_{j',\varphi(k') }} && T_{\varphi(k')}
   } 
   \]
approximately commute within $F$ and $G$ respectively, that is, 
\[\sigma_{i,\psi(k)}\simeq_F \beta_k\circ \tau_{\varphi(j),k}\circ\alpha_j\circ\sigma_{i,j}\quad \text{and} \quad \tau_{i,\varphi(k')}\simeq_G \alpha_{k'}\circ\sigma_{\psi(j'),k'}\circ\beta_{j'}\circ\tau_{i,j'}.
\]

We say that $(\alpha_i,\beta_i)_i$ is a \emph{two-sided approximate intertwining}.\end{dfn}

\begin{rmk}
Each of the sequences $(\alpha_i)_i$ and $(\beta_i)_i$ that define a two-sided approximate intertwining induce a one-sided approximate intertwining. 
\end{rmk}

Throughout the paper, when considering a finite set $F$ for comparison of $\Cu$-morphisms, we will often need to construct a larger finite set $\tilde{F}$ which is \emph{finer} than $F$ in the following sense. 
\begin{dfn}
Let $F,\tilde{F}$ be finite subsets of a $\Cu$-semigroup and let $n\in\N$. We will say that $\tilde{F}$ is an \emph{$n$-refinement} of $F$, or that $\tilde{F}$ \emph{refines} $F$ \emph{$n$-times}, if
\begin{itemize}
\item[(i)] $F\subseteq \tilde{F}$. 
\item[(ii)] For any $f',f\in F$ such that $f'\ll f$, there exist $n$ elements $g_1,\ldots,g_n\in \tilde{F}$ such that $f'\ll g_1\ll\ldots\ll g_n\ll f$.
\end{itemize}
\end{dfn}

Note that, for any $n\geq 1$ and any finite set $F$ of a $\Cu$-semigroup, we can always find an $n$-refinement of $F$.

\begin{rmk}
One of the reasons why the previous notion is needed is that $\simeq_F$ is not a transitive relation, that is, $\alpha\simeq_F\beta\simeq_F\gamma$ does not imply $\alpha\simeq_F\gamma$. Instead, what we do have is that, if $\tilde{F}$ is an $n$-refinement of $F$, then $\alpha\simeq_{\tilde{F}}\alpha_1\simeq_{\tilde{F}}\ldots\simeq_{\tilde{F}}\alpha_n\simeq_{\tilde{F}}\beta$ implies $\alpha\simeq_F\beta$.
\end{rmk}

\begin{thm}\label{thm:2SidIntCu}
Let $(S_i,\sigma_{i,j})_{i\in\N}$ and $(T_i,\tau_{i,j})_{i\in\N}$ be two inductive sequences in $\Cu$. Assume that there exists a two-sided approximate intertwining $(\alpha_i\colon S_i\longrightarrow T_{\varphi(i)},\beta_i\colon T_i\longrightarrow S_{\psi(i)})_i$. 

Then there exists a $\Cu$-isomorphism $\alpha\colon S\cong T$ induced by $(\alpha_i)_i$ whose inverse $\beta\colon T\cong S$ is induced by $(\beta_i)_i$.
\end{thm}
\begin{proof}
Our approach is similar to that of \cite[Theorem 3.16]{C22}, but we proceed with additional care since our setting is more general.

Arguing as in the proof of \autoref{prop:induced}, we know that for any $i\in\N$ the sequences $(\tau_{\varphi(j),\infty}\circ\alpha_j\circ\sigma_{i,j})_j$ and $(\sigma_{\psi(j'),\infty}\circ\beta_{j'}\circ\tau_{i,j'})_{j'}$ are Cauchy, and we denote their respective limits by $\eta_i\colon S_i\longrightarrow T$ and $\nu_i\colon T_i\longrightarrow S$. Furthermore, these limits induce $\Cu$-morphisms $\alpha\colon S\longrightarrow T$ and $\beta\colon T\longrightarrow S$.
In order to show that $\alpha$ and $\beta$ are inverses of one another, it suffices to show that $\beta\circ \eta_i=\sigma_{i,\infty}$ and $\alpha\circ \nu_i=\tau_{i,\infty}$ for any $i\in \N$.

Let $F$ be a finite subset of $S_i$ and let $\tilde{F}$ be a $2$-refinement of $F$.
Since $\eta_i$ is the limit of $(\tau_{\varphi(j),\infty}\circ\alpha_j\circ\sigma_{i,j})_j$, we know that there exists $j\geq i$ big enough such that $\eta_i\simeq_{\tilde{F}} \tau_{\varphi(j),\infty}\circ\alpha_j\circ \sigma_{i,j}$. Post-composing with $\beta$, we obtain
\begin{equation}\label{eq:TwoSid1}
\beta\circ\eta_i\simeq_{\tilde{F}} \nu_{\varphi(j)}\circ\alpha_j\circ \sigma_{i,j}.
\end{equation}
Consider $\tilde{G}:=\alpha_j\circ\sigma_{i,j}(\tilde{F})\subseteq T_{\varphi(j)}$. Since $\nu_{\varphi(j)}$ is the limit of $(\sigma_{\psi(k),\infty}\circ\beta_{k}\circ\tau_{\varphi(j),k})_{k}$, we know that there exists  $k\geq \varphi(j)$ big enough such that $\nu_{\varphi(j)}\simeq_{\tilde{G}} \sigma_{\psi(k),\infty}\circ\beta_{k}\circ\tau_{\varphi(j),k}$. Precomposing by $\alpha_j\circ\sigma_{i,j}$, this implies
\begin{equation}\label{eq:TwoSid2}
\nu_{\varphi(j)}\circ \alpha_j\circ\sigma_{i,j} \simeq_{\tilde{F}} \sigma_{\psi(k),\infty}\circ\beta_{k}\circ\tau_{\varphi(j),k}\circ \alpha_j\circ\sigma_{i,j}.
\end{equation}
Finally, since $(\alpha_i,\beta_i)_i$ is a two-sided approximate intertwining, we also have
\begin{equation}\label{eq:TwoSid3}
 \sigma_{\psi(k),\infty}\circ\beta_{k}\circ\tau_{\varphi(j),k}\circ \alpha_j\circ\sigma_{i,j}\simeq_{\tilde{F}}\sigma_{i,\infty}
\end{equation}
whenever $j$ and $k$ are big enough.

It follows from the construction of $\tilde{F}$ and a combination of (\ref{eq:TwoSid1})-(\ref{eq:TwoSid3}) that $\beta\circ\eta_i\simeq_F \sigma_{i,\infty}$. Since this holds for any finite subset, we must have $\beta\circ\eta_i = \sigma_{i,\infty}$. Therefore, we get $\beta\circ\alpha=\id_S$. The fact that $\alpha\circ\beta=\id_T$ follows from a symmetric argument.
\end{proof}

\begin{prg}[\textbf{Comparison and approximate intertwinings in $\Cus$}]\label{prg:CompAppIntCus} Note that none of the proofs above uses the fact that the ordered monoids under consideration are positively ordered. Thus, all the results in this section are still valid for the larger category $\Cus$ introduced in \cite{C23c} (loosely, this is the category of not necessarily positively ordered $\Cu$-semigroups).

Many refinements of the Cuntz semigroup have $\Cus$ as their target category, and thus are amenable to the techniques developed here. We predict that this will play an important role when one such variant of the Cuntz semigroup is used to classify morphisms between certain $\CatCa$-algebras.
\end{prg}

\subsection{Fra\"{i}ss\'{e} Categories of Cuntz semigroups}
\label{subsec:FraisseCu}
We are now ready to introduce a version of Fra\"{i}ss\'{e} Theory for abstract Cuntz semigroups, analogous to that of \cite{K13}. As stated earlier, we use  finite-set comparison of $\Cu$-morphisms to bypass the need of enriched metrics.


\begin{dfn}\label{dfn:cufraissedfn}
Let $\cc$ be a subcategory of $\Cu$. We say that $\cc$
\begin{itemize}
 \item[(JEP$_{\Cu}$)] satisfies the \emph{(Cuntz) joint embedding property} if, for any $A_1,A_2\in \cc$, there exists $B\in \cc$ such that both $\Hom_\cc(A_1,B)$ and $\Hom_\cc(A_2,B)$ are nonempty.
 \item[(NAP$_{\Cu}$)] satisfies the \emph{(Cuntz) near amalgamation property} if, for any pair of $\cc$-morphisms $\alpha_1\in \Hom_\cc(A,B_1)$ and $\alpha_2\in \Hom_\cc(A,B_2)$, and any finite subset $F\subseteq A$, there exist $C\in \cc$ and $\cc$-morphisms $\beta_1\in \Hom_\cc(B_1,C)$ and $\beta_2\in \Hom_\cc(B_2,C)$ such that $\beta_1\circ\alpha_1\simeq_F \beta_2\circ\alpha_2$.
 \item[(SEP$_{\Cu}$)] is \emph{(Cuntz) separable} if there exists a countable dominating subcategory $\mathfrak{s}\subseteq \cc$, that is,
 \begin{itemize}
 \item[$\bullet$] any object $S\in \mathfrak{s}$ is a countably-based $\Cu$-semigroup.
 \item[$\bullet$] the set of $\mathfrak{s}$-morphisms is countable.
  \item[$\bullet$] for any $A\in \cc$ there exist $S\in \mathfrak{s}$ such that $\Hom_\cc(A,S)$ is nonempty.
  \item[$\bullet$] for any $\cc$-morphism $\sigma\colon S\longrightarrow A$ with $S\in \mathfrak{s}$ and any finite subset $F\subseteq S$, there exists $T\in\mathfrak{s}$ and $\alpha\in\Hom_\mathfrak{c}(A,T)$ and $\tau\in\Hom_\mathfrak{s}(S,T)$ such that $\alpha \circ\sigma\simeq_F \tau$. \end{itemize}
\end{itemize}

We say that $\cc$ is a \emph{(Cuntz) Fra\"{i}ss\'{e} category} if $\cc$ satisfies (JEP$_{\Cu}$), (NAP$_{\Cu}$) and (SEP$_{\Cu}$).
\end{dfn}

Next, we define a notion of \emph{(Cuntz) Fra\"{i}ss\'{e} sequences} and show that any (Cuntz) Fra\"{i}ss\'{e} category admits such a sequence, which is unique up to two-sided approximate intertwining.

\begin{dfn}\label{dfn:cufraisseseqdfn}
Let $\cc$ be a subcategory of $\Cu$. An inductive sequence $(S_i,\sigma_{i,j})_{i\in\N}$ is called a \emph{(Cuntz) Fra\"{i}ss\'{e} sequence} if
\begin{itemize}
\item[(i)] Every $S_i$ is a countably-based $\Cu$-semigroup.
\item[(ii)] For any finite subset $F\subseteq S_i$ and any $\cc$-morphism $\alpha\colon S_i\longrightarrow C$, there exists a $\cc$-morphism $\beta_F\colon C\longrightarrow S_j$ for some $j\geq i$ such that $\beta_F\circ\alpha\simeq_F\sigma_{i,j}$.
\end{itemize} 
\end{dfn}

\begin{thm}[Existence and Uniqueness]
\label{thm:existenceuniqueness}
Let $\mathfrak{c}\subseteq \Cu$ be a Fra\"{i}ss\'{e} category. Then $\mathfrak{c}$ admits a Fra\"{i}ss\'{e} sequence which is unique up to two-sided approximate intertwining in $\cc$. 
\end{thm}
\begin{proof}
The proof that such a sequence exists is analogous to that of \cite[Theorem 3.3]{K13}. We reproduce the proof here with our language of finite-set comparison for the sake of completeness. In contrast, the uniqueness part of the proof differs from \cite{K13}.

[Existence]  Without loss of generality, we may assume that the countable dominating subcategory $\mathfrak{s}\subseteq \cc$ satisfies the joint embedding property and the near amalgamation property. 

Now consider the partially ordered set $\mathcal{S}:=\{\text{finite inductive sequences in } \mathfrak{s}\}$ with the end-extension order, i.e. $(S_i,\sigma_{i,j})_{i,j\leq n}\leq (T_i,\tau_{i,j})_{i,j\leq m}$ in $\mathrm{S}$ whenever $n\leq m$ and $(T_i,\tau_{i,j})_{i,j\leq n}=(S_i,\sigma_{i,j})_{i,j\leq n}$. For any $S\in \mathfrak{s}$, fix a basis $B:=\bigcup\limits_k B_k$ such that $(B_k)_k$ is a $\subseteq$-increasing sequence of finite sets of $S$.
For any $\mathfrak{s}$-morphism $\alpha\colon S\longrightarrow T$ and $n,k\in\N$, we let $D_{n,\alpha,k}$ be the subset of $\mathrm{S}$ of all elements $(S_i,\sigma_{i,j})_{i,j\leq m}$ satisfying the following
\begin{itemize}
 	\item[$\bullet$] $m>n$.
 	\item[$\bullet$] $\Hom_\mathfrak{s}(S,S_i)\neq \emptyset$ for some $i$.
  	\item[$\bullet$] if $S=S_n$, then there exist $j>n$ and $\beta\in\Hom_\mathfrak{s}(T,S_j)$ such that $\beta\circ\alpha\simeq_{B_k}\sigma_{n,j}$.
\end{itemize}
Using the joint embedding property and the near amalgamation property, it is readily checked that all sets of the form $D_{n,\alpha,k}$ are cofinal in $\mathcal{S}$ with respect to the end-extension order, that is, for any triple $n,\alpha ,k$ and any $s\in\mathcal{S}$ there exists  $d\in D_{n,\alpha,k}$ such that $s\leq d$.

Arguing similarly as in the proof of \cite[Theorem 3.3]{K13}, let us fix an ordering $\varphi\colon \N\longrightarrow\{(n,\alpha,k)\mid n,k\in\N,\, \alpha\in\mathfrak{s}\}$. Then, we can use the Rasiowa-Sikorski lemma to find a $\leq$-increasing sequence $(c_l)_l$ where $c_l\in D_{\varphi(l)}$. Note that the supremum of $(c_l)_l$ is in fact a well-defined inductive sequence in $\mathfrak{s}$, which we write as $(S_i,\sigma_{i,j})_i$. By construction, $(S_i,\sigma_{i,j})_i$ is Fra\"{i}ss\'{e} for $\mathfrak{s}$. Finally, arguing again as in the proof of \cite[Theorem 3.3]{K13}, we deduce that $(S_i,\sigma_{i,j})_i$  is in fact a Fra\"{i}ss\'{e} sequence for $\cc$.

[Uniqueness] Let $(S_i,\sigma_{i,j})_i$ and $(T_i,\tau_{i,j})_i$ be two Fra\"{i}ss\'{e} sequences in $\cc$. We are going to recursively construct a two-sided approximate intertwining between them.

 First, recall that all the $\Cu$-semigroups involved are countably-based (by definition). Therefore, for each $i$, we can fix two $\subseteq$-increasing sequences $(B^i_n)_n, (C^i_n)_n$ of finite subsets of $S_i,T_i$ respectively, such that their unions over $n$ are sup-dense in $S_i$ and $T_i$. Now, using the joint embedding property first, and then the Fra\"{i}ss\'{e} sequence property twice, we construct the maps $\beta_0$ and $\alpha_{\psi(0)}$ which make the following diagram approximately commute.
\[
   \xymatrix{
     S_0\ar[rr]\ar[rd]_{}&&S_{\psi(0)}\ar[ddr]^{\alpha_{\psi(0)}} \\
     &V_0\ar[ru] \ar@{}[u]|-{\simeq_{B_0^0}} &\\
     T_0\ar[rrr] \ar[ru]\ar@{-->}@/^{-1pc}/[rruu]_{\beta_0}&&  \ar@{}[uu]|-(.3){\simeq_{C_0^0}} &T_{\varphi(0)}
 }  
\]
Our aim is to construct strictly increasing map $\psi ,\varphi\colon \N\longrightarrow \N$ together with large enough subsets $\tilde{B}_{\psi(i)}, \tilde{C}_{\varphi(i)}$ of $S_{\psi(i)},T_{\varphi(i)}$ and $\cc$-morphisms $\beta_{\varphi(i)}\colon T_{\varphi(i)}\longrightarrow S_{\psi(i+1)}$ and $\alpha_{\psi(i+1)}\colon S_{\psi(i+1)}\longrightarrow T_{\varphi(i+1)}$ producing the following approximately commutative diagram.
\[
   \xymatrix{
&&\ldots\ar[r]
 	&S_{\psi(i)}\ar[rr]\ar[ddr]_{\alpha_{\psi(i)}} & \ar@{}[dd]|-(.2){\simeq_{\tilde{B}_{\psi(i)}}} &S_{\psi(i+1)}\ar[rdd]^{\alpha_{\psi(i+1)}}
 \ar[rr]&&\ldots
\\
&&& && &&&
&&
\\
&&\ldots\ar[rr]
	&& T_{\varphi(i)}\ar[rr]\ar[ruu]^(.4){\beta_{\varphi(i)}}&  \ar@{}[uu]|-(.3){\simeq_{\tilde{C}_{\varphi(i)}}} &T_{\varphi(i+1)}
\ar[r]&\ldots
}  
\] 
To do this, we proceed by induction from the initial data $B_0^0,C_0^0, \beta_0,\alpha_{\psi(0)}$. Assume that 
\begin{itemize}
\item the numbers $\psi(0),\varphi(0),\ldots,\psi(i),\varphi(i)$
\item the finite sets $B_0^0, C_0^0,\ldots,\tilde{B}_{\psi(i-1)}, \tilde{C}_{\varphi(i-1)}$
\item  the $\cc$-morphisms $\beta_0,\alpha_{\psi(0)},\ldots,\beta_{\varphi(i-1)},\alpha_{\psi(i)}$
\end{itemize}
have been constructed for some $i\geq 0$. (By convention, we have fixed $\psi(-1)=\varphi(-1)=0$ and  $\tilde{B}_{0}:=B_0^0, \tilde{C}_{0}:=C_0^0$.)

In what follows, a \emph{path} is any $\Cu$-morphism in the above diagram that can be expressed as the composition of finitely many maps among $\sigma_{j,k}$, $\tau_{j,k}$, $\alpha_{\psi (j)}$ and $\beta_{\varphi (j)}$. Let us start by choosing $\tilde{B}_{\psi(i)}\subseteq S_{\psi(i)}$ and $\tilde{C}_{\varphi(i)}\subseteq T_{\varphi(i)}$ such that
\begin{itemize}
 	\item[(i)] $B_{\psi(i)}^{\psi(i)}\subseteq \tilde{B}_{\psi(i)}$ and $C_{\varphi(i)}^{\varphi(i)}\subseteq \tilde{C}_{\varphi(i)}$.
 	\item[(ii)] $\tilde{B}_{\psi(i)}$  refines $\left\{ \pi (b)\in S_{\psi(i)} \mid b\in \bigcup_{0<l\leq i}\left(\tilde{B}_{\psi(i-l)}\cup \tilde{C}_{\varphi(i-l)}\right),\, \pi\text{ a path} \right\}$.
	\item[(iii)] $\tilde{C}_{\varphi(i)}$ refines $\left\{ \pi (c)\in T_{\varphi(i)} \mid c\in \bigcup_{0<l\leq i}\left(\tilde{B}_{\psi(i-l)}\cup \tilde{C}_{\varphi(i-l)}\right),\, \pi\text{ a path} \right\}$.
 \end{itemize}

We apply successively the Fra\"{i}ss\'{e} sequence property twice. First, we obtain an index $\psi(i+1)>\psi(i)$ together with a $\cc$-morphism $\beta_{\varphi(i)}\colon T_{\varphi(i)}\longrightarrow S_{\psi(i+1)}$ such that 
\[
	\beta_{\varphi(i)}\circ \alpha_{\psi(i)}\simeq_{\tilde{B}_{\psi(i)}} \sigma_{\psi(i),\psi(i+1)}
\]
and then we get an index $\varphi (i+1)>\varphi (i)$ together with a $\cc$-morphism $\alpha_{\psi (i+1)}\colon S_{\psi (i+1)}\longrightarrow T_{\varphi (i+1)}$ such that
\[
	\alpha_{\psi (i+1)}\circ\beta_{\varphi(i)}\simeq_{\tilde{C}_{\varphi(i)}} \tau_{\varphi (i),\varphi (i+1)}
\]
which finishes the inductive argument.

We will now check that the sequences of $\cc$-morphisms that we have just constructed induce a two-sided approximate intertwining. Let us first prove the following.

\textbf{Claim.} Let $i,l\in\N$. For any elements $b^-,b,b^+\in \tilde{B}_{\psi(i)}$ such that $ b^-\ll b\ll b^+$, we have
\[
\left\{
\begin{array}{ll}
\sigma_{\psi(i),\psi(i+l+1)}(b^-)\leq \beta_{\varphi(i+l)}\circ\tau_{\varphi(i),\varphi(i+l)}\circ\alpha_{\psi(i)}(b^+)\\
 \beta_{\varphi(i+l)}\circ\tau_{\varphi(i),\varphi(i+l)}\circ\alpha_{\psi(i)}(b^-)\leq\sigma_{\psi(i),\psi(i+l+1)}(b^+)
\end{array}
\right.
\]
\emph{Proof of the Claim}. Note that $\sigma_{\psi(i),\psi(i+1)}(b^-)\leq \beta_{\varphi(i)}\circ\alpha_{\psi(i)}(b)$. Moreover, it follows from the construction of $\tilde{B}_{\psi(i+1)}$ (see (ii) above) that there exists $b_2\in \tilde{B}_{\psi(i+1)}$ such that $\sigma_{\psi(i),\psi(i+1)}(b^-)\ll b_2 \ll \beta_{\varphi(i)}\circ\alpha_{\psi(i)}(b)$. Thus, we have that
\begin{align*}
\sigma_{\psi(i),\psi(i+2)}(b^-)&\leq \beta_{\varphi(i+1)}\circ\alpha_{\psi(i+1)}(b_2)\\
&\leq \beta_{\varphi(i+1)}\circ\alpha_{\psi(i+1)}(\beta_{\varphi(i)}\circ\alpha_{\psi(i)}(b)).
\end{align*}
Proceeding successively in this fashion, we obtain
\begin{equation}\label{eq:FraUni1}
	\sigma_{\psi(i),\psi(i+l+1)}(b^-)\leq (\beta_{\varphi(i+l)}\circ\alpha_{\psi(i+l)})\circ \ldots \circ(\beta_{\varphi(i)}\circ\alpha_{\psi(i)})(b).
\end{equation}
A similar argument involving the pair $\alpha_{\psi(i)}(b)$ and $\alpha_{\psi(i)}(b^+)$ shows that 
\begin{equation}\label{eq:FraUni2}
(\alpha_{\psi(i+l)}\circ\beta_{\varphi(i+l-1)})\circ \ldots \circ(\alpha_{\psi(i+1)}\circ\beta_{\varphi(i)})(\alpha_{\psi(i)}(b))\leq \tau_{\varphi(i),\varphi(i+l)}(\alpha_{\psi(i)}(b^+)).
\end{equation}
Post-composing (\ref{eq:FraUni2}) by $\beta_{\varphi(i+l)}$ and combining it with (\ref{eq:FraUni1}) gives us the first inequality of the claim. The other inequality follows from a symmetric argument, which proves the claim.

Finally, let $F$ be a finite subset of $S_i$. From the construction of the $\tilde{B}_{\psi(i)}$'s (see (i) above), there exists an index $i_F\geq i$ such that, for any $f',f\in F$ with $f'\ll f$ and any $j\geq i_F$, we can find $b^-,b,b^+,\in \tilde{B}_{\psi(j)}$ such that $\sigma_{i,\psi(j)}(f')\ll b^-\ll b\ll b^+\ll \sigma_{i,\psi(j)}(f)$. Applying the claim, we get
\[
\left\{
\begin{array}{ll}
\sigma_{i,\psi(k+1)}(f')\leq \beta_{\varphi(k)}\circ\tau_{\varphi(j),\varphi(k)}\circ\alpha_{\psi(j)}\circ\sigma_{i,\psi(j)}(f)\\
\beta_{\varphi(k)}\circ\tau_{\varphi(j),\varphi(k)}\circ\alpha_{\psi(j)}\circ\sigma_{i,\psi(j)}(f')
\leq \sigma_{i,\psi(k+1)}(f)\end{array}
\right.
\]
for any $j\geq i_F$ and any $k\geq j+1$. In other words, $\sigma_{i,\psi(k+1)}\simeq_F \beta_{\varphi(k)}\circ\tau_{\varphi(j),\varphi(k)}\circ\alpha_{\psi(j)}\circ\sigma_{i,\psi(j)}$. Setting $\alpha_j':=\alpha_{\psi(j)}\circ\sigma_{j,\psi(j)}$ and $\beta_k':=\beta_{\varphi(k)}\circ\tau_{k,\varphi(k)}$, we obtain 
\[
	\sigma_{i,\psi(k+1)}\simeq_F \beta_{k}'\circ\tau_{\varphi(j),k}\circ\alpha_{j}'\circ\sigma_{i,j}.
\]
The second property of \autoref{dfn:atwosided} is proved by using a symmetric argument, which implies that $(\alpha_j')_j$ and $(\beta_k')_k$ are two-sided approximate intertwining.
 \end{proof}

The following technical lemmas about Fra\"{i}ss\'{e} sequences will be needed to obtain our version of  \autoref{thm:fraisselim}.

\begin{lma}[Universality]
\label{lma:universality}
Let $\cc\subseteq \Cu$ be a Fra\"{i}ss\'{e} category. Let $(S_i,\sigma_{i,j})_{i\in\N}$ be a Fra\"{i}ss\'{e} sequence and let  $(T_i,\tau_{i,j})_{i\in\N}$ be a $\cc$-inductive sequence of countably-based $\cc$-objects. 

Then, there exists a one-sided approximate intertwining $(\alpha_i\colon T_i\longrightarrow S_{\varphi(i)})_i$.
\end{lma}
\begin{proof}
We are going to construct the desired one-sided approximate intertwining recursively following a similar argument to that of \autoref{thm:existenceuniqueness}.

Recall that all the $\Cu$-semigroups involved are countably-based by definition. Therefore, for each $i$, we can fix a $\subseteq$-increasing sequence $(B^i_n)_n$ of finite subsets of $T_i$ such that its unions over $n$ is sup-dense in $T_i$. Now, using the joint embedding property together with the fact that $(S_i,\sigma_{i,j})_i$ is a Fra\"{i}ss\'{e} sequence, we construct a $\cc$-morphism $\alpha_0\colon T_0\longrightarrow S_{\varphi(0)}$.
\[
   \xymatrix{
    	& T_0\ar@{.>}@/^{1,5pc}/[dd]^{\alpha_0}\ar[d]_{}\\
	&V_0\ar[d]\\
 	S_0\ar[r]\ar[ru]&S_{\varphi(0)}
 }  
\]

Our aim is to construct a strictly increasing map $\varphi\colon \N\longrightarrow \N$ together with large enough subsets $\tilde{B}_{i}$ of $T_{i}$ and $\cc$-morphisms $\alpha_{i+1}\colon T_{i+1}\longrightarrow S_{\varphi(i+1)}$ such that $\sigma_{\varphi(i),\varphi(i+1)}\circ \alpha_i\simeq_{\tilde{B_i}} \alpha_{i+1}\circ \tau_{i,i+1}$.
To do this, we proceed by induction from the initial data $B_0^0,\alpha_0$. Assume that 
\begin{itemize}
\item the numbers $\varphi(0),\ldots,\varphi(i)$
\item the finite sets $B_0^0,\tilde{B}_{1}, \ldots,\tilde{B}_{i-1}$
\item  the $\cc$-morphisms $\alpha_0,\ldots,\alpha_{i}$
\end{itemize}
have been constructed for some $i\geq 0$. (By convention, we have fixed $\tilde{B}_{-1}:=B_0^0$.)

Let us start by choosing $\tilde{B}_{i}\subseteq T_{i}$ such that
\begin{itemize}
 	\item[(i)] $B_{i}^{i}\subseteq \tilde{B}_{i}$.
	\item[(ii)] $\tilde{B}_{i}$ is a $2$-refinement of $\left\{ \tau_{i-l,i}(b)\in T_i \mid b\in \bigcup_{0<l\leq i}\tilde{B}_{i-l} \right\}$.
\end{itemize}
We then use the near amalgamation property to construct $\cc$-morphisms $\xi_{i+1}^u,\theta_{i+1}$ such that the quadrilateral in the diagram below approximately commutes within $\tilde{B}_{i}$. Lastly, we use the Fra\"{i}ss\'{e} sequence property to get a $\cc$-morphism $\xi_{i+1}^d$ such that the triangle underneath approximately commutes within $\alpha_{i}(\tilde{B}_{i})$. Define $\alpha_{i+1}:=\xi_{i+1}^d\circ\theta_{i+1}$. 
\[
   \xymatrix{
	T_{i}\ar[rr] \ar[dd]_{\alpha_{i}}& \ar@{}[dd]|-(.3){\simeq_{\tilde{B}_{i}}} &T_{i+1}\ar@{.>}@/^{2,5pc}/[dd]^{\alpha_{i+1}}\ar@{.>}[d]^{\theta_{i+1}}\\
	&&	V_{i+1}\ar@{.>}[d]^{\xi^d_{i+1}}\\
 	S_{\varphi(i)}\ar[rr] \ar@{.>}[rru]_{}^{\xi^u_{i+1}}&&S_{\varphi(i+1)} \ar@{}[llu]|-(.15){\simeq_{\alpha_{i}(\tilde{B}_{i})}}
 }  
\]
Hence, we have obtained a sequence of $\cc$-morphisms $(\alpha_i\colon T_i\longrightarrow S_{\varphi(i)})_i$, and we are left to show that this is a one-sided approximate intertwining.

\textbf{Claim.} Let $i,j\in\N$. For any elements $b',b^-_1,b^+_1,b\in \tilde{B}_{i}$ such that $b'\ll b^-_1\ll b^+_1\ll b$, we have
\[
\left\{
\begin{array}{ll}
\sigma_{\varphi(i),\varphi(i+j)}\circ\alpha_{i}(b')\leq\alpha_{i+j}\circ\tau_{i,i+j}(b)\\
\alpha_{i+j}\circ\tau_{i,i+j}(b')\leq \sigma_{\varphi(i),\varphi(i+j)}\circ\alpha_{i}(b)
\end{array}
\right.
\]
\emph{Proof of the claim}. Note that
\[
\sigma_{\varphi(i),\varphi(i+1)}\circ\alpha_i(b')\leq
 \xi^d_{i+1}\circ\xi^u_{i+1}\circ\alpha_i(b^-_1)\leq
 \xi^d_{i+1}\circ\theta_{i+1}\circ\tau_{i,i+1}(b^+_1)=
 \alpha_{i+1}\circ\tau_{i,i+1}(b^+_1).
\]
From the construction of the $\tilde{B}_{i}$'s (see (ii) above), we know that we can find $b^-_2,b^+_2\in \tilde{B}_{i+1}$ such that $\tau_{i,i+1}(b^+_1)\ll b^-_2 \ll b^+_2\ll \tau_{i,i+1}(b)$. Therefore, we have 
\[
\sigma_{\varphi(i+1),\varphi(i+2)}\circ\alpha_{i+1}\circ\tau_{i,i+1}(b^+_1)\leq\xi^d_{i+2}\circ\xi^u_{i+2}\circ\alpha_{i+1}(b^-_2)\leq \alpha_{i+2}\circ\tau_{i+1,i+2}(b^+_2).
\]

Proceeding successively in this fashion, we obtain elements $b^+_1,b^+_2,\ldots ,b^+_{j}$ such that $b^+_l\in \tilde{B}_{i+l-1}$, and $b^+_l\ll \tau_{i,i+l-1}(b)$ for any $l\leq j$. We compute
\begin{align*}
\sigma_{\varphi(i),\varphi(i+1)}\circ\alpha_i(b')&\leq \alpha_{i+1}\circ\tau_{i,i+1}(b_1^+)\\
\sigma_{\varphi(i+1),\varphi(i+2)}\circ\alpha_{i+1}(\tau_{i,i+1}(b_1^+))&\leq \alpha_{i+2}\circ\tau_{i+1,i+2}(b^+_2)\\
\sigma_{\varphi(i+2),\varphi(i+3)}\circ\alpha_{i+2}(\tau_{i+1,i+2}(b^+_2))&\leq \alpha_{i+3}\circ\tau_{i+2,i+3}(b^+_3)\\
&\,\,\,\,\vdots\\
\sigma_{\varphi(i+j-1),\varphi(i+j)}\circ\alpha_{i+j-1}(\tau_{i+j-2,i+j-1}(b^+_{j-1}))&\leq \alpha_{i+j}\circ\tau_{i+j-1,i+j}(b^+_j).
\end{align*} 
This proves the first inequality of the claim. The other inequality is shown using a symmetric argument.

Finally let $F$ be a finite subset of $T_i$. From the construction of the $\tilde{B}_{i}$'s (see (i) above), there exists an index $i_F\geq i$ such that, for any $j\geq i_F$ and any pair $f',f\in F$ with $f'\ll f$, we can find $b',b^-_1,b^+_1,b,\in \tilde{B}_{j}$ such that $\tau_{i,j}(f')\ll b'\ll b^-_1\ll b^+_1\ll b\ll \tau_{i,j}(f)$. Applying the claim, we readily obtain
\[
\left\{
\begin{array}{ll}
\sigma_{\varphi(j),\varphi(k)}\circ\alpha_{j}\circ\tau_{i,j}(f')\leq \alpha_{k}\circ\tau_{i,k}(f)\\
 \alpha_{k}\circ\tau_{i,k}(f')\leq\sigma_{\varphi(j),\varphi(k)}\circ\alpha_{j}\circ\tau_{i,j}(f)
\end{array}
\right.
\]
for any $j\geq i_F$ and any $k> j$. In other words, $\alpha_{k}\circ\tau_{i,k}\simeq_F \sigma_{\varphi(j),\varphi(k)}\circ\alpha_{j}\circ\tau_{i,j}$, as required.
\end{proof}

\begin{lma}[Homogeneity]
\label{lma:homogeneity}
Let $\cc\subseteq \Cu$ be a Fra\"{i}ss\'{e} category and let $(S_i,\sigma_{i,j})_{i\in\N}$ be a Fra\"{i}ss\'{e} sequence. 
Then, for any $\cc$-morphisms $\alpha\colon C\longrightarrow S_l$, $\beta\colon C\longrightarrow S_l$ and any finite subset $F\subseteq C$, there exists a two-sided approximate intertwining $(\eta_i\colon S_i\longrightarrow S_{\varphi(i)},\nu_i\colon S_i\longrightarrow S_{\psi(i)})_{i\geq l}$ such that
\[
\sigma_{l,\psi(i)}\circ\alpha\simeq_F \nu_i\circ\sigma_{l,i}\circ\beta \quad \text{ and }\quad \sigma_{l,\varphi(i)}\circ\beta\simeq_F\eta_i\circ\sigma_{l,i}\circ\alpha
 \]
for any $i\geq l$.
\end{lma}
\begin{proof}
We will build the two-sided approximate intertwining following the structure of the previous proofs.

First, let us consider a $4$-refinement $\tilde{F}$ of $F$. (That is, $\tilde{F}$ is a finite subset of $C$ that contains $F$ and is such that, for any $f',f\in F$ with $f'\ll f$, there exist $g',g^-,g^+,g\in\tilde{F}$ satisfying $f'\ll g'\ll g^-\ll g^+\ll g\ll f$.) Then, using the near amalgamation property together with the Fra\"{i}ss\'{e} sequence property, we construct $\cc$-morphisms $\nu_l,\eta_{\psi(l)}$ such that the following diagram approximately commutes.
\[
   \xymatrix{
&S_{l}\ar[rr] \ar[rd]_{} \ar@{}[dd]|-{\simeq_{\tilde{F}}} & \ar@{}[d]|-(.3){\simeq_{\alpha(\tilde{F})}} &S_{\psi(l)}\ar[rdd]^{\eta_{\psi(l)}}&\\
C\ar[ru]^{\alpha}\ar[rd]_{\beta}&&V_{l}\ar[ru]&\\
&S_{l}\ar@{-->}@/^{-1pc}/[rruu]_{\nu_{l}}\ar[ru]\ar[rrr]&&  \ar@{}[uu]|-(.3){\simeq_{\beta(\tilde{F})}} &S_{\varphi(l)}
 }  
\]

Using the ideas and techniques from \autoref{thm:existenceuniqueness}, it is readily verified that
\[
\left\{
\begin{array}{ll}
\sigma_{l,\varphi(l)}\circ\beta(f') &\leq\sigma_{l,\varphi(l)}\circ\beta(g')\leq \eta_{\psi(l)}\circ\sigma_{l,\psi(l)}\circ\alpha(g)\leq \eta_{\psi(l)}\circ\sigma_{l,\psi(l)}\circ\alpha(f)\\
\sigma_{l,\psi(l)}\circ\alpha(f') &\leq \sigma_{l,\psi(l)}\circ\alpha(g')\leq\nu_l\circ\beta(g)\leq \nu_l\circ\beta(f)
\end{array}
\right.
\]

Finally, from the initial data $\tilde{B}_l:=\alpha(\tilde{F}),\tilde{C}_l:=\beta(\tilde{F}), \nu_l,\eta_{\psi(l)}$, one can construct a two-sided approximate intertwining following the proof of \autoref{thm:existenceuniqueness} (starting at $l$ instead of $0$). Such an intertwining will enjoy the desired properties.
\end{proof}

We now have all the tools that we need to obtain Fra\"{i}ss\'{e} limits in the category $\Cu$. Let us first introduce the almost factorization property adapted to our setting.

\begin{dfn}\label{dfn:AFPCu}
Let $\cc\subseteq \dd$ be an inclusion of categories in $\Cu$ such that $\dd$ has inductive limits. We say that the inclusion $\cc\subseteq \dd$ has the \emph{(Cuntz) almost factorization property} if, for any $C\in \cc$, any $\cc$-inductive system $(S_i,\sigma_{i,j})_{i,j\in I}$ with $\dd$-limit $(S,\sigma_{i,\infty})_{i}$, any  $\dd$-morphism $\alpha\colon C\longrightarrow S$, and any finite subset $F\subseteq C$, there exist an index $i_F$ and a $\cc$-morphism $\alpha_F\colon C\longrightarrow S_{i_F}$ satisfying $\sigma_{i_F,\infty}\circ \alpha_F\simeq_F \alpha$.
\end{dfn}

Adapting the definition of the Ind-completion (see, for example, \cite[Chapter~VI]{J82}), we define:

\begin{dfn}\label{dfn:CompletionCu}
Let $\cc$ be a subcategory of $\Cu$. The \emph{completion of $\cc$}, denoted by $\overline{\cc}$, is the subcategory of $\Cu$ whose
\begin{itemize}
 \item[(i)] objects are $\Cu$-limits of inductive sequences in $\cc$.
 \item[(ii)] morphisms are induced by \emph{some} one-sided approximate intertwining. 
 
More specifically, a $\Cu$-morphism $\alpha\colon S\longrightarrow T$ between $\overline{\cc}$-objects $S,T$ is a $\overline{\cc}$-morphism if for any $\cc$-inductive sequence $(S_i,\sigma_{i,j})_{i}$ whose $\Cu$-limit objects is $S$, there exists a $\cc$-inductive sequence $(T_i,\tau_{i,j})_{i}$ whose $\Cu$-limit object is $T$ together with a one-sided approximate intertwining $(\alpha_i\colon S_i\longrightarrow T_{\varphi(i)})_i$ in $\cc$ which induces $\alpha$ in the sense of \autoref{prop:induced}.
\end{itemize} 
\end{dfn}

\begin{rmk}\label{rmk:IndCSubCat}
The following two properties of $\overline{\cc}$ are readily verified. (For example, they can be adapted from the results and references from  \cite[Chapter~VI]{J82}.)
\begin{itemize}
\item[(i)] The category $\overline{\cc}$ is a well-defined subcategory of $\Cu$ containing $\cc$ as a subcategory.
\item[(ii)] Any inductive sequence in $\cc$ has an inductive limit in $\overline{\cc}$ which coincides with its inductive limit in $\Cu$.
 In particular, any object $S\in\overline{\cc}$ can be written as the $\overline{\cc}$-limit object of an inductive sequence  in $\cc$.
\end{itemize}

Further, note that $\cc\subseteq\overline{\cc}$ may not satisfy the almost factorization property. However, this will be the case for all our examples.
\end{rmk}


\begin{thm}\label{prp:FraisseLimCu}
Let $\cc$ be a Fra\"{i}ss\'{e} category. Then any Fra\"{i}ss\'{e} sequence $(S_i,\sigma_{i,j})_i$ has a $\overline{\cc}$-limit $(S,\sigma_{i,\infty})_i$ such that
\begin{itemize}
 \item[(i)] $S$ is unique up to isomorphism, that is, $S$ does not depend on the Fra\"{i}ss\'{e} sequence chosen.
 \item[(ii)] The set $\Hom_{\overline{\cc}}(D,S)$ is nonempty whenever $D$ is countably-based.
\end{itemize}

Assume that $\cc$ is contained in a category $\dd$ where every $\cc$-inductive sequence has a limit, and where every $\dd$-object is the limit of a $\cc$-sequence. If $\cc\subseteq\dd$ satisfies the almost factorization property, then
\begin{itemize}
 \item[(iii)] For any $C\in\cc$, any $\alpha,\beta\in \Hom_{\dd}(C,S)$ and any finite set $F\subseteq C$, there exists a $\dd$-isomorphism $\eta_F\colon S\overset{\cong}{\longrightarrow} S$ such that $\eta_F\circ\alpha \simeq_F\beta$.
\end{itemize}
\end{thm}
\begin{proof}
(i) follows immediately from \autoref{thm:existenceuniqueness} together with \autoref{thm:2SidIntCu}, while (ii) follows as a combination of \autoref{lma:universality} and \autoref{prop:induced}.

To see (iii), let  $\alpha,\beta\in \Hom_{\dd } (C,S)$ and let $F\subseteq C$ be a finite subset. Construct a $3$-refinement $\tilde{F}$ of $F$. Let $(S_i,\sigma_{i,j})_i$ be a Fra\"{i}ss\'{e} sequence. Then, by the almost factorization property, there exist morphisms $\alpha_{\tilde{F}},\beta_{\tilde{F}}\colon C\to S_l$ such that $\alpha\simeq_{\tilde{F}}\sigma_{l,\infty}\circ\alpha_{\tilde{F}}$ and $\beta\simeq_{\tilde{F}}\sigma_{l,\infty}\circ\beta_{\tilde{F}}$. Using \autoref{lma:homogeneity} and \autoref{thm:2SidIntCu}, we see that there exists an isomorphism $\eta_F\colon S\longrightarrow S$ satisfying the desired condition.
\end{proof}

\begin{prg}[\textbf{Fra\"{i}ss\'{e} Categories of $\Cus$-semigroups}] Following the discussion from \autoref{prg:CompAppIntCus}, we note that all the results above also do not use the fact that the underlying ordered monoids have a positive order. Thus, we have in fact developed a Fra\"{i}ss\'{e} theory for $\Cus$-semigroups.
\end{prg}


\subsection{\texorpdfstring{$\CatCa$-algebras and Fra\"{i}ss\'{e} categories of Cuntz semigroups}{C*-algebras and Fraisse categories of Cuntz semigroups}}
\label{subsec:ClassMorCat}
In this last subsection we study under which assumptions the functor $\Cu$ induces a Fra\"{i}ss\'{e} category of Cuntz semigroups when applied to a Fra\"{i}ss\'{e} category of separable $\CatCa$-algebras. A natural (but rather strong) assumption to consider is that $\Cu$ classifies $^*$-homomorphisms of the Fra\"{i}ss\'{e} category $\cc\subseteq \CatCa$ at hand. We will see that, under an additional mild assumption, this is sufficient to deduce that $\Cu(\cc)$ is a Fra\"{i}ss\'{e} category. 
We will conclude with some remarks on  the link between these Fra\"{i}ss\'{e} categories, where we discuss a weak converse of our result and ways to considerably weaken the statement. 

Let us start by recalling the definition of classifying morphisms. (See, for example, \cite{R12} or \cite{C21a} for more details.) 

\begin{dfn}
Let $\cc$ and $\dd$ be subcategories of separable $\CatCa$-algebras. We say that the functor $\Cu$ \emph{classifies $^*$-homomorphisms} from $\cc$ to $\dd$ if, for any $A$ in $\cc$, any $B$ in $\dd$ and any scaled $\Cu$-morphism $\alpha\colon\Cu (A)\longrightarrow\Cu (B)$, there exists a $^*$-homomorphism $\chi\colon A\longrightarrow B$, unique up to approximate unitary equivalence, such that $\Cu (\chi)=\alpha$. 

We will say that $\Cu$ \emph{classifies $^*$-homomorphisms of $\cc$} whenever $\cc=\dd$. 
\end{dfn}

\begin{lma}
\label{lma:cstarcu}
Let $\cc$ be a category of separable $\CatCa$-algebras which is either full or replete\footnote{Most of the Fra\"{i}ss\'{e} categories that we will consider are not full (for example, one usually considers injective maps). However, they will all be replete.}. Assume that $\Cu$ classifies $^*$-homomorphisms of $\cc$. Then $\Cu(c)$ is a subcategory of $\Cu$.
\end{lma} 

\begin{proof}
In both cases, the fact that the identity of every $\Cu(c)$-object is a $\Cu(c)$-morphism, and that both domain and codomain of a $\Cu(c)$-morphism are in $\Cu(c)$ is immediate. The non-trivial part is to check that $\Cu(\cc)$ is closed under composition of morphisms. 

Let $\varphi\colon A\longrightarrow B$ and $\psi\colon B'\longrightarrow C$ be $^*$-homomorphisms such that $\Cu(B)=\Cu(B')$. Write $\alpha:=\Cu(\psi)\circ\Cu(\varphi)\colon\Cu(A)\longrightarrow \Cu(C)$. Note that $\alpha$ is still a scaled morphism, that is, it maps the class of a strictly positive element in $A$ to a class below a strictly positive element in $C$. 

Assume that $\cc$ is full. Since $\Cu$ classifies $^*$-homomorphisms of $\cc$, we know that there exists a $^*$-homomorphism $\chi_\alpha\colon A\longrightarrow C$, which is a $\cc$-morphism by fullness of $\cc$, such that $\Cu(\chi_\alpha)=\alpha$. In other words, $\alpha\in\Cu(\cc)$. 

Now assume that $\cc$ is replete. Since $\Cu$ classifies $^*$-homomorphisms of $\cc$ and $\Cu(B)=\Cu(B')$, we know that there exists $^*$-isomorphism $\chi\colon B\cong B'$ lifting $\id_{\Cu(B)}$. Since $\cc$ is replete, we also know that $\chi$ is in fact a $\cc$-morphism. Now define $\chi_\alpha:=\psi\circ\chi\circ\varphi\colon A\longrightarrow C$. From construction, $\chi_\alpha$ is a $\cc$-morphism and $\Cu(\chi_\alpha)=\alpha$. In other words, $\alpha\in\Cu(\cc)$. 
\end{proof}

\begin{thm}
\label{thm:cstarcu}
Let $\cc$ be a category of separable $\CatCa$-algebras which is either full or replete. Assume that $\Cu$ classifies $^*$-homomorphisms of $\cc$. 

If $\cc$ is a Fra\"{i}ss\'{e} category whose Fra\"{i}ss\'{e} limit is $A$, then  $\Cu(\cc)$ is a Fra\"{i}ss\'{e} category whose Fra\"{i}ss\'{e} limit is $\Cu(A)$. 
\end{thm}
\begin{proof}
The joint embedding property immediately follows applying the functor $\Cu$. 

Let us prove the near amalgamation property in $\Cu(\cc)$. First, we note the following fact.

\textbf{Fact}. Let $A$ be a $\CatCa$-algebra and let $F\subseteq \Cu(A)$ be a finite subset. Then there exists a small enough $\varepsilon>0$ such that, for any $[a],[b]\in F$ with $[a]\ll [b]$, then $[a]\ll [(b-\varepsilon)_+]$. 

Let $\Cu(\phi_1)\colon \Cu(A)\longrightarrow \Cu(B_1)$ and $\Cu(\phi_2)\colon \Cu(A')\longrightarrow \Cu(B_2)$ be $\cc$-morphisms with $\Cu(A)=\Cu(A')$. Using the same arguments as above, we may assume $A=A'$. Let $F\subseteq \Cu(A)$ be finite and let $\varepsilon >0$ be the constant given by the previous fact. Now, using the near amalgamation property in $\cc$, we know that there exist $\cc$-morphisms $\psi_1\colon B_1\longrightarrow C$ and $\psi_2\colon B_2\longrightarrow C$ such that $d_{\CatCa}(\psi_1\circ\phi_1,\psi_2\circ\phi_2)<\varepsilon$. We are left to show that $\Cu(\psi_1\circ\phi_1)\simeq_F \Cu(\psi_2\circ\phi_2)$. Let $[a],[b]\in F$ be such that $[a]\ll [b]$. We know that $\Vert \psi_1\circ\phi_1(b)-\psi_2\circ\phi_2(b)\Vert<\varepsilon$ and that $[a]\ll [(b-\varepsilon)_+]$ which implies that 
\[
\left\{
\begin{array}{ll}
\Cu(\psi_1\circ\phi_1)([a])\ll\Cu(\psi_1\circ\phi_1)([(b-\varepsilon)_+])\leq \Cu(\psi_2\circ\phi_2)([b])\\
\Cu(\psi_2\circ\phi_2)([a])\ll\Cu(\psi_2\circ\phi_2)([(b-\varepsilon)_+])\leq \Cu(\psi_1\circ\phi_1)([b])
\end{array}
\right.
\]
as desired.

That $\Cu (\cc )$ is separable is shown similarly, by also using the fact that any $\CatCa$-algebra in $\cc$ is separable, and thus gives rise to a countably-based $\Cu$-semigroup.

We deduce that $\Cu(\cc)$ is a Fra\"{i}ss\'{e} category. Further, given any Fra\"{i}ss\'{e} sequence in $\cc$, we can use that $\Cu$ classifies $^*$-homomorphisms of $\cc$ to show that the induced sequence in $\Cu$ is Fra\"{i}ss\'{e} in $\Cu (\cc)$. Thus, by continuity of the  functor $\Cu$, it is readily checked that $\Cu(A)$ is the Fra\"{i}ss\'{e} limit of $\Cu(\cc)$.
\end{proof}

Some remarks are in order:

\begin{rmk}\label{rmk:endingpartc}

(i) Assume that $\cc$ is such that $\Cu$ classifies $^*$-homomorphisms from $\cc$ to stable rank one $\CatCa$-algebras or, more generally, to a category $\dd$ closed under ultraproducts. Then, it follows from \cite[Theorem~3.3.1]{R12} that the previous theorem has a weak converse: $\Cu (\cc )$ is Fra\"{i}ss\'{e} if and only if $\cc$ is Fra\"{i}ss\'{e} with respect to approximate unitary equivalences (in the sense of \autoref{rmk:FraisseAUE}).

\noindent (ii) The assumption of classification of $^*$-homomorphisms is rather strong in general. For instance, it is proved in \cite{C23b} that the functor $\Cu$ does not classify $^*$-homomorphisms of circle algebras. Nevertheless, we do not use the full force of the assumption, neither in \autoref{lma:cstarcu} nor in \autoref{thm:cstarcu}. Explicitly, one only needs to assume the following much weaker condition:
\begin{itemize}
\item[] \emph{For every $A,A',B,B'\in\cc$ such that $\Cu (A)=\Cu(A')$ and $\Cu(B)=\Cu(B')$, and any $\cc$-morphism $\varphi\colon A\to B$, there exists a $\cc$-morphism $\phi\colon A'\to B'$ such that $\Cu (\phi )=\Cu(\varphi )$.}
\end{itemize} 

Note that, in particular, this holds whenever our category has a single object. (See \autoref{rmk:AltPrfCantPseud}.)

To obtain a weak converse (as in (i)), the additional property that one needs is:
\begin{itemize}
\item[] \emph{For every $\CatCa$-algebra $A$ in $\cc$, any finite set $F\subseteq A$ and any $\varepsilon>0$ there exists a finite subset $G\subseteq\Cu (A)$ such that, whenever two $\cc$-morphisms $\varphi_1,\varphi_2\colon A\to B$ satisfy $\Cu (\varphi_1)\simeq_G\Cu(\varphi_2)$, then there exists $u\in\tilde{B}$ such that $\Vert u^*\varphi_1 (x)u - \varphi_2 (x)\Vert<\varepsilon$ for every $x\in F$.}
\end{itemize}

\noindent (iii) A number of examples of $\CatCa$-algebraic Fra\"{i}ss\'{e} subcategories $\cc$ have injective $^*$-homo\-morphisms as morphisms. Note that the theorem above does not imply that $\cc$ induces a category $\Cu(\cc)$ whose morphisms are order-embeddings, since injective $^*$-homomorphisms do not generally induce injective $\Cu$-morphisms.

For example, the diagonal map $\C\oplus\C\longrightarrow M_2 (\C )$ is injective, but the induced $\Cu$-morphism maps both $[(1,0)]$ and $[(0,1)]$ to $[1\oplus 0]$. Conversely, a $^*$-homomorphism that induces an order-embedding may not be injective.
\end{rmk}


\section{Examples}\label{sec:examples}
In this section we exhibit natural examples of (Cuntz) Fra\"{i}ss\'{e} categories together with their Fra\"{i}ss\'{e} limit. This allows us to deduce several generic properties about the $\Cu$-semigroups at play. 
More explicitly, we show that the Cuntz semigroup of any UHF-algebra and that the Cuntz semigroup of the universal AF-algebra are the Fra\"{i}ss\'{e} limits of some well-chosen Fra\"{i}ss\'{e} categories. We also show that there are countably many Fra\"{i}ss\'{e} categories of elementary $\Cu$-semigroups whose Fra\"{i}ss\'{e} limits are not purely infinite, non-stably finite, simple $\Cu$-semigroups. Finally, we prove that both $\Cu$-semigroups $\Lsc(X,\overline{\N})$, where $X$ is either the Cantor set $2^\N$ or the pseudo-arc $\mathbb{P}$, can also be written as Fra\"{i}ss\'{e} limits.

In all our examples the inclusion $\cc\subseteq\overline{\cc}$ satisfies the almost factorization property.

\subsection{\texorpdfstring{Dimension $\Cu$-semigroups of infinite type as Fra\"{i}ss\'{e} limits}{Dimension Cu-semigroups of infinite type as Fraisse limits}}\label{subsec:A} In what follows, we show that the Cuntz semigroup of any UHF-algebra arises as a Fra\"{i}ss\'{e} limit. Following \cite{APT18}, recall that a $\Cu$-semigroup $S$ is said to be \emph{simplicial} whenever $S\cong \overline{\N}^r$ for some $r\in\N$, and that an inductive limit of simplicial $\Cu$-semigroups is called a \emph{dimension $\Cu$-semigroup}. 

Let $p$ be a prime number and consider the semigroup $S_p:=\N[\frac{1}{p}]\,\sqcup\,(0,\infty]$, where the mixed sum and mixed order are defined as follows: 
\begin{itemize}
 \item[(+)] $x_c+x_s=2x_s$ for any pair $x_c=k/p^l \in \N[\frac{1}{p}]$ and $x_s=k/p^l\in (0,\infty]$.
 \item[($\leq $)] retaining the same notation, $x_s\leq x_c\ll x_c \leq x_s+\varepsilon$ for any $\varepsilon>0$.
\end{itemize}
It is well-known that $S_p$ is the Cuntz semigroup of the UHF-algebra $M_{p^\infty}$. (See, e.g. \cite[Proposition~7.4.3]{APT18}.)


\begin{prg} \emph{The category $\mathfrak{s}_p$} is the category whose (single) object is $\overline{\N}$ and whose morphisms are powers of $p$. (That is, a $\mathfrak{s}_p$-morphism is given by $\overline{\N}\overset{\times p^k}\longrightarrow \overline{\N}$ for some $k\in\N$.)
\end{prg}
\begin{thm}\label{prp:SpCuFrais}
The category $\mathfrak{s}_p$ is Fra\"{i}ss\'{e} and its limit is $S_p$.
 \end{thm}
 
 \begin{proof}
The joint embedding property and separability are clear, while the  amalgamation property follows from the fact that the composition in $\Hom_{\mathfrak{s}_p}(\overline{\N},\overline{\N})$ is commutative. Therefore $\mathfrak{s}_p$ is a Fra\"{i}ss\'{e} category. 
 
 Let us now show that the sequence given by $(\overline{\N},\overset{\times p}\longrightarrow)_n$ is Fra\"{i}ss\'{e}. For any $\mathfrak{s}_p$-morphism $\alpha\colon\overline{\N}\overset{\times p^k}\longrightarrow \overline{\N}$ there exists a large enough index  (consisting of $k+1$ steps further from the domain of $\alpha$) such that the following diagram commutes
\[
   \xymatrix{
     \ldots\ar[r] & \overline{\N}\ar[rd]_{\alpha}\ar[r] &\ldots\ar[r] & \overline{\N}\ar[r] &\ldots \\
   && \overline{\N}\ar[ru]_{\times p} &&
   } 
  \]
It follows that the sequence is Fra\"{i}ss\'{e} and, hence, its limit $S_p$ is the Fra\"{i}ss\'{e} limit of $\mathfrak{s}_p$.
\end{proof}

\begin{rmk}
 An analogous statement to that of \autoref{prp:SpCuFrais} works for the Cuntz semigroup of any UHF-algebra. However, such semigroups may not be as easily computed as those of infinite type.
 Also, let us remark that we could have deduced the previous result from \autoref{thm:cstarcu}.
\end{rmk}

\subsection{\texorpdfstring{The universal dimension $\Cu$-semigroup as a Fra\"{i}ss\'{e} limit}{The universal dimension Cu-semigroup as a Fraisse limit}}\label{subsec:B}

Based on the work done in \cite{GK20}, we now exhibit a \emph{universal dimension $\Cu$-semigroup} $\mathcal{S}$ as a Fra\"{i}ss\'{e} limit of a well-chosen Fra\"{i}ss\'{e} category $\mathfrak{s}_{\dim}$ containing all simplicial $\Cu$-semigroups. As a matter of fact, we will see that $\mathcal{S}$ is the concrete Cuntz semigroup of the universal $\AF$-algebra $\mathcal{A}$ constructed in \cite{GK20}. Therefore, $\mathcal{S}$ enjoys analogous properties to those of $\mathcal{A}$, but in the category $\Cu$. Before we begin, we need to recall and prove some results about \emph{retractions} in the category $\Cu$.

\begin{prg}[\textbf{Retractions}]
We aim to define a Fra\"{i}ss\'{e} category $\mathfrak{s}_{\dim}$ that will be built out of retractable $\Cu$-morphisms.

Recall that an ordered monoid morphism $\alpha\colon S\longrightarrow T$ between two $\Cu$-semigroups is said to be a \emph{generalized $\Cu$-morphism} if $\alpha$ preserves suprema of increasing sequences.
\end{prg}
\begin{dfn}[{\cite[Definition~3.14]{TV22}}]
Let $S,T$ be $\Cu$-semigroups. We say that $S$ is a \emph{retract} of $T$ if there exists a $\Cu$-morphism $\iota\colon S\longrightarrow T$ together with a generalized $\Cu$-morphism $\rho\colon T\longrightarrow S$ such that $\rho\circ\iota=\id_S$. 

We say that $\iota$ is \emph{retractable} and that $\rho$ \emph{retracts} $\iota$.
\end{dfn}

Following \cite{CRS10}, a submonoid $I$ of a $\Cu$-semigroup $S$ is said to be an \emph{ideal} if $I$ is downward-hereditary and closed under suprema of increasing sequences. Given any ideal $I$ of $S$, one can construct the quotient $\Cu$-semigroup $S/I$; see \cite[Lemma~5.1.2]{APT18}. Ideals and quotients of a $\CatCa$-algebra $A$ are in bijective correspondence with the ideals and quotients of its Cuntz semigroup $\Cu (A)$; see \cite[Section~5]{APT18}.

\begin{prop}
Let $S$ be a $\Cu$-semigroup.
\begin{itemize}
 \item[(i)] If $S$ is countably-based and satisfies (O5)-(O7), then any ideal of $S$ is a retract of $S$.
 \item[(ii)] If $(S,\sigma_{i,\infty})_{i}$ is the limit of an inductive sequence $(S_i,\sigma_{i,j})_{i}$ such that $\sigma_{i,i+1}$ is retractable for each $i\in\N$, then $\sigma_{i,\infty}$ is retractable for any $i\in\N$.
\end{itemize}
\end{prop}
\begin{proof}
(i) Let $I$ be an ideal of a countably-based $\Cu$-semigroup $S$ and let $\iota\colon I\longrightarrow S$ be the canonical order-embedding.
Observe that $I$ is a countably-based $\Cu$-semigroup. Thus, $I$ has a greatest element that we denote by $\infty_I$. Using \cite[Theorem~2.4]{APRT21}, we know that the infimum $x\wedge \infty_I$ exists for any $x\in S$. Further, such an infimum is always in $I$, since $x\wedge\infty_I\leq \infty_I$.

Now consider the map $\rho\colon S\longrightarrow I$ that sends $x\mapsto x\wedge\infty_I$. It follows directly from \cite[Theorem~2.5]{APRT21} that $\rho$ is a generalized $\Cu$-morphism that satisfies $\rho\circ\iota =\id_I$, as desired. 

(ii) Let $(S_i,\sigma_{i,j})_{i}$ be an inductive sequence with retractable maps, and let $(S,\sigma_{i,\infty})_i$ be its inductive limit.
 Denote by $\rho_{i+1,i}$ the retract of $\sigma_{i,i+1}$. For each pair $i\leq j$, let $\rho_{j,i}\colon S_j\longrightarrow S_i$ be the composition $\rho_{i+1,i}\circ\ldots\circ \rho_{j,j-1}$, which trivially retracts $\sigma_{i,j}$. 
 
 Now fix $i\in\N$ and let $s\in S$. Using the characterization of sequential inductive limits in the category $\Cu$ (see e.g. \cite[Section 2.1]{R12}), we know that there exists a sequence $(s_n)_{n\geq i}$ with $s_n\in S_n$ such that $\sigma_{n,n+1}(s_n)\ll s_{n+1}$ and $\sup_{n\geq i} \sigma_{n,\infty}(s_n)=s$. Applying the retract at each step gives us that $s_n\leq \rho_{{n+1},n}(s_{n+1})$ and, hence, $\rho_{n,i}(s_n)\leq \rho_{{n+1},i}(s_{n+1})$ for any $n\geq i$. 
 
 Set $\rho_{\infty ,i} (s):=\sup_{n\geq i} \rho_{n,i}(s_n)$. A standard argument shows that $\rho_{\infty ,i}\colon S\longrightarrow S_i$ does not depend on the sequence $(s_n)_{n\geq i}$ chosen. This also proves that $\rho_{\infty ,i}$ is additive, preserves the order, and suprema of increasing sequences. (See e.g. the argument in \cite[Lemma 3.12]{C22} or in \cite[Lemma~7.3]{AABPV23}.)
 
Finally, let $x\in S_i$ and let $(x_n)_n$ be a $\ll$-increasing sequence in $S_i$ with supremum $x$. For each $n\geq i$, set $s_n:=\sigma_{i,n}(x_n)$. Since $\rho_{\infty,i}$ does not depend on the sequence $(s_n)_n$ chosen, we have
 \[
  \rho_{\infty ,i} (\sigma_{i,\infty} (x))= \sup_n \rho_{n,i}(s_n)
  = \sup_n \rho_{n,i}(\sigma_{i,n}(x_n))= \sup_n x_n = 
  x
 \]
which ends the proof.
\end{proof}

The following is based on \cite[Proposition~3.8]{GK20}.

\begin{prop}\label{prop:RetId}
Let $\cc$ be a subcategory of $\Cu$ closed under direct sums. Let $(S_i,\sigma_{i,j})_{i}$ be an  inductive sequence in $\cc$ with $\overline{\cc}$-limit $(S,\sigma_{i,\infty })_i$. 

Then, there exists an inductive sequence $(T_i,\tau_{i,j})_{i}$ in $\cc$ with $\overline{\cc}$-limit $(T,\tau_{i,\infty })_i$ such that $\tau_{i,j}$ is retractable for every pair $i\leq j$, and such that $S\cong T/\! J$ for some ideal $J$ of $T$.
\end{prop}
\begin{proof}
Let $i\in\N$ and set $T_i:=S_0\oplus \ldots\oplus S_{i-1}\oplus S_i$. Let $\tau_{i,i+1}\colon T_i\longrightarrow T_{i+1}$ be the $\Cu$-morphism defined by
\[
 \tau_{i,i+1}(s_0,\ldots,s_{i-1} ,s_i):=
 (s_0,\ldots,s_{i-1}, s_i ,\sigma_{i,i+1}(s_i)).
\]
Note that $\tau_{i,i+1}\colon T_i\longrightarrow T_{i+1} $ is retractable by the projection $\pi_{i+1,i}\colon T_{i+1}\longrightarrow T_i$ onto the first $i$ components of $T_i$. More particularly, we have
\[
 \pi_{i+1,i}(s_0,\ldots,s_{i} ,s_{i+1})=
 (s_0,\ldots ,s_i).
\]
It follows that every composition $\tau_{i,j}:=\tau_{j-1,j}\circ\ldots\circ \tau_{i,i+1}$ is also retractable for any $j\geq i+1$.

Let $(T,\tau_{i,\infty})_i$ be the limit of the inductive system of $(T_{i},\tau_{i,j})_i$ and set
\[
 J_0:=\{ x\in T\mid x=\tau_{i,\infty } ( s_0,\ldots ,s_{i-1},0 )\text { for some } i\in\N \text{ and }s_0\in S_0,\ldots,s_{i-1}\in S_{i-1} \}.
\]
Note that $J_0$ is a submonoid of $T$. Thus, we can construct its sup-closure $J:=\overline{J_0}^{\rm sup}$. (See e.g. \cite[Definition~4.6]{TV21}.) Further, given $y\in J$ and $x\in T$ such that $x\leq y$, take $x'\in T$ such that $x'\ll x$. Since $J$ is the closure of $J_0$ and $x'\ll x$, there exist $i\in\N$, $t\in T_i$, and $(s_0,\ldots , s_{i-1},0)\in T_i$ such that 
\[
	x'\ll \tau_{i,\infty }(t)\ll x,\quad\text{and}\quad 
	t\ll (s_0,\ldots , s_{i-1},0).
\]
In particular, we get $t\in S_0\oplus\ldots\oplus S_{i-1}\oplus 0$ and, hence, $\tau_{i,\infty }(t)\in J_0$.

This shows that $x$ can be written as the supremum of elements in $J_0$ or, equivalently, that $x$ is in $J$. Thus, $J$ is an ideal of $T$. It is now clear that $S\cong T/\! J$.
\end{proof}


\begin{prg} \emph{The category $\mathfrak{s}_{\dim}$} is the category whose objects are the simplicial $\Cu$-semigroups and whose morphisms are retractable $\Cu$-morphisms.
\end{prg}

\begin{thm}\label{prp:KubisCuFraisse}
The category $\mathfrak{s}_{\dim}$ is Fra\"{i}ss\'{e} and its limit $\mathcal{S}$  is the Cuntz semigroup of the universal $\AF$-algebra.
\end{thm}
\begin{proof}
First, note that any retractable $\Cu$-morphism between simplicial $\Cu$-semigroups can always be retracted by a $\Cu$-morphism. Indeed, let $\iota\colon S\longrightarrow T$ be a $\Cu$-morphism between simplicial $\Cu$-semigroups, and let $\rho\colon T\longrightarrow S$ be a generalized $\Cu$-morphism such that $\rho\circ\iota=\id_{S}$.

Define $\rho'\colon T\longrightarrow S$ by $\rho'(x):=\rho (x\wedge\infty\iota (1_S))$. This map is the composition of generalized $\Cu$-morphisms, and so it itself is a generalized $\Cu$-morphism. To see that it preserves the $\ll$-relation, simply take $x\in T$ such that $x\ll x$. Then, one gets $x\wedge\infty\iota (1_S)\ll x\wedge\infty\iota (1_S)$ and, consequently, that $x\wedge\infty\iota (1_S)\leq n\iota (1_S)$ for some $n\geq 1$. This implies $\rho'(x)\leq n1_S$, which is equivalent to $\rho'(x)\ll \rho'(x)$, as desired.

The result now follows from \cite{GK20} and \autoref{thm:cstarcu}, combined with the well-known fact that the Cuntz semigroup classifies $^*$-homomorphisms of finite dimensional $\CatCa$-algebras.
\end{proof}

\begin{cor}\label{prp:QuotUniDimCu}
Let $\mathcal{S}$ be the universal dimension $\Cu$-semigroup and let $S$ be a (countably-based) dimension $\Cu$-semigroup. Then, there exists a surjective $\Cu$-morphism $\mathcal{S}\relbar\joinrel\twoheadrightarrow S$.
\end{cor}
\begin{proof}
By \autoref{prop:RetId} we know that $S$ is isomorphic to a quotient of the form $T/J$, where $T$ is the limit of some  inductive system in $\mathfrak{s}_{\dim}$. Further, by universality of the Fra\"{i}ss\'{e} limit, there exists a retractable $\Cu$-morphism $\iota\colon T\longrightarrow\mathcal{S}$ whose retract $\rho\colon \mathcal{S}\relbar\joinrel\twoheadrightarrow T$ is a $\Cu$-morphism. The composition of $\rho$ with the quotient map $T\relbar\joinrel\twoheadrightarrow T/J$ gives the desired $\Cu$-morphism.
\end{proof}

\subsection{Elementary Fra\"{i}ss\'{e} categories}\label{subsec:C}

As defined in \cite[8.1]{TV22b}, a $\Cu$-semigroup $S$ is said to be \emph{elementary} if $S$ simple and contains a minimal, nonzero element that is finite\footnote{This differs slightly from the definition given in \cite{APT18}, and was adjusted to not include the Cuntz semigroup of simple, purely infinite $\CatCa$-algebras.}. If $S$ satisfies (O5) and (O6), it follows from \cite[Proposition~5.1.19]{APT18} that $S\cong E_{n}:=\{ 0,1,\ldots ,n,\infty\}$ for some $n\geq 1$. It is well known that no such $\Cu$-semigroup is the Cuntz semigroup of a $\CatCa$-algebra\footnote{A quick proof of this fact goes as follows: Assume for the sake of contradiction that there exists $A$ with $\Cu (A)$ isomorphic to $E_n$ for some $n$. Then, $A$ must be simple and weakly purely infinite and, consequently, purely infinite. The Cuntz semigroup of any simple, purely infinite $\CatCa$-algebra is isomorphic to $\{ 0,\infty\}$, a contradiction.}.

The following lemma shows that the $\Cu$-morphisms between elementary $\Cu$-semigroups are well understood. By an \emph{order-embedding} between $\Cu$-semigroups we will always mean a $\Cu$-morphism that is also an order-embedding.

\begin{lma}\label{prp:MorElCu}
Let $n,m\in\N$. A map $\alpha\colon \{ 0,1,\ldots ,n,\infty\}\longrightarrow \{ 0,1,\ldots ,m,\infty\}$ is
\begin{itemize}
\item[(i)] a $\Cu$-morphism if and only if $(n+1)\alpha (1)=\infty$ and $\alpha (k1)=k\alpha (1)$ for every $k\leq n$.
\item[(ii)] an order-embedding if and only if $\alpha (1)\neq 0$ and $n\alpha (1)\neq\infty$.
\end{itemize}
\end{lma}

In particular, there exist choices of natural numbers $n,m\in\N$ such that there is no or only one order-embedding from $E_n$ to $E_m$. For example, we know that $\alpha (1)=k1$ must satisfy $m/(n+1)<k\leq m/n$. Thus, if $m=n(n+1)$, we get that the only order-embedding from $E_n$ to $E_m$ is $\alpha(1)=(n+1)1$.


\begin{prg} \emph{The category $\mathfrak{e}_{\infty}$} is the category whose objects are all elementary $\Cu$-semigroups satisfying (O5) and (O6) and whose morphisms are all nonzero $\Cu$-morphisms.
\end{prg}

\begin{prop}\label{prp:InfElFraisse}
The category $\mathfrak{e}_{\infty}$ is  Fra\"{i}ss\'{e}.
\end{prop}
\begin{proof}
The category $\mathfrak{e}_{\infty}$ contains countably many objects, and the morphisms between them are given by multiplication. It follows that $\mathfrak{e}_{\infty}$ is separable.

Further, given $n_1,n_2\in\N$, it follows from \autoref{prp:MorElCu} that the maps $\alpha_1\colon E_{n_1}\longrightarrow E_{n_1n_2}$ and $\alpha_2\colon E_{n_2}\longrightarrow E_{n_1n_2}$ given by $1\mapsto n_2 1$ and $1\mapsto n_1 1$ respectively are order-embeddings (in particular, nonzero $\Cu$-morphisms). This shows that the category $\mathfrak{e}_{\infty}$ has the joint embedding property.

Finally, given any pair of nonzero $\Cu$-morphisms $\alpha_1\colon E_{n}\longrightarrow E_{m}$ and $\alpha_2\colon E_{n}\longrightarrow E_{m}$, one can simply consider the map $\beta\colon E_m\longrightarrow E_m$ given by $\beta (1):=\infty$. This $\Cu$-morphism satisfies $\beta\circ \alpha_1=\beta\circ\alpha_2$, which implies that the category $\mathfrak{e}_{\infty}$ is  Fra\"{i}ss\'{e}.
\end{proof}

\begin{cor}\label{cor:einftyfraisse}
The Fra\"{i}ss\'{e} limit of $\mathfrak{e}_{\infty}$ is  $\{ 0,\infty\}$, that is, the Cuntz semigroup of any purely infinite simple $\CatCa$-algebra.
\end{cor}
\begin{proof}
Let $(S_i,\sigma_{i,j})_i$ be a Fra\"{i}ss\'{e} sequence of $\mathfrak{e}_{\infty}$. For any $i\in\N$, let $\alpha\colon S_i\longrightarrow \{ 0,1,\infty\}$ be the morphism that maps every nonzero element to $\infty$. By definition, there exists $\beta\colon \{ 0,1,\infty\}\longrightarrow S_j$ such that $\sigma_{i,j}=\beta\circ\alpha$. In other words, every element in $S_i$ becomes idempotent eventually.

The only nonzero, simple, idempotent $\Cu$-semigroup is $\{ 0,\infty\}$. Thus, since the morphisms in the sequence are nonzero, the Fra\"{i}ss\'{e} limit of $\mathfrak{e}_{\infty}$ is isomorphic to $\{ 0,\infty\}$.
\end{proof}

\begin{rmk}
The subcategory obtained by considering all order-embeddings instead of all nonzero $\Cu$-morphisms is not Fra\"{i}ss\'{e}, since it fails to satisfy the joint embedding property. For example, one can check that the morphisms $\alpha_1,\alpha_2\colon\{0,1,\infty\}\longrightarrow \{0,\ldots ,6,\infty\}$ given by $\alpha_1(1)=4\cdot 1$ and $\alpha_2(1)=5\cdot 1$ cannot be amalgamated.
\end{rmk}


\begin{prg} \emph{The category $\mathfrak{e}_n$} (for a fixed $n\geq 1$) is the category whose objects are are all the elementary semigroups of the form $E_{n^k}$ for some $k\in\N$, and whose morphisms are all the order-embeddings given by multiplication by a power of $n$.
\end{prg}

Note that not all maps given by powers of $n$ are $\Cu$-morphisms. As an example, set $n=2$. Then, we know from \autoref{prp:MorElCu} that multiplying by $2$ does not give rise to a $\Cu$-morphism $E_2\longrightarrow E_8$. In fact, it follows from \autoref{prp:MorElCu} that for any given pair $k\leq s$ there exists a unique order-embedding $E_{n^k}\longrightarrow E_{n^s}$ given by a power of $n$, namely $1\mapsto n^{s-k}$.

\begin{prop}\label{prp:PElFraisse}
The category $\mathfrak{e}_n$ is Fra\"{i}ss\'{e} for any $n\in\N$.
\end{prop}
\begin{proof}
That the category $\mathfrak{e}_n$ is separable and satisfies the joint embedding property is proven similarly as in the proof of \autoref{prp:InfElFraisse}.

To see that the category $\mathfrak{e}_n$ has amalgamation, let  $\alpha_1\colon E_{n^k}\longrightarrow E_{n^{s_1}}$ and $\alpha_2\colon E_{n^k}\longrightarrow E_{n^{s_2}}$ be $\mathfrak{e}_n$-morphisms. From the joint embedding property, we may assume that $s_1=s_2$.
Now, we know from the comments above that these two maps must be the same. Thus, the category $\mathfrak{e}_n$ has amalgamation.
\end{proof}

Recall that a $\Cu$-semigroup $S$ is \emph{stably finite} if $x\neq x+y$ for every nonzero $y\in S$ whenever $x\ll z$ for some $z\in S$.

Also, extending the definition of \cite{KR00}, let us say that $S$ is \emph{($n$-)weakly purely infinite} if $nx=2(nx)$ for every $x\in S$.

\begin{cor}\label{cor:enfraisse}
The Fra\"{i}ss\'{e} limit $\mathcal{E}_n$ of $\mathfrak{e}_n$ is a simple, non-stably finite, not weakly purely infinite $\Cu$-semigroup whose order is total.
\end{cor}
\begin{proof}
The Fra\"{i}ss\'{e} sequence of $\mathfrak{e}_n$ consists of simple $\Cu$-semigroups, so its limit must be simple. Further, since each $\Cu$-semigroup has a total order, so does the limit.

A simple $\Cu$-semigroup is not stably finite if and only if its greatest element, denoted by $\infty$, satisfies $\infty\ll\infty$. The Fra\"{i}ss\'{e} limit of $\mathfrak{e}_n$ admits a nonzero $\Cu$-morphism $\alpha$ from $\{ 0,1,\ldots ,n,\infty\}$. In particular, $\alpha (\infty )=\infty$ and, since $\alpha$ preserves the $\ll$-relation, one has $\infty\ll\infty$ in the limit. This shows that the limit is not stably finite.

Finally, the limit cannot be weakly purely infinite. Indeed, for any $k\in\N$ we can use \autoref{prp:FraisseLimCu} to find an order-embedding $\alpha$ from $\{ 0,1,\ldots ,n^k,\infty \}$ to the limit. Then, $\alpha (1)$ satisfies $l\alpha(1)\neq (l-1)\alpha (1)$ for any $l\leq n^k$. This shows that there is no global bound on the idempotency of the elements, which implies that the limit is not weakly purely infinite.
\end{proof}

\begin{rmk}
Arguing similarly as in \autoref{prp:SpCuFrais}, we can deduce that, given a prime number $p$, the Fra\"{i}ss\'{e} limit $\mathcal{E}_p$ of $\mathfrak{e}_p$ is a truncated version of the dimension semigroup of infinite type $S_p=\N[\frac{1}{p}]\,\sqcup\,(0,\infty]$. Explicitly, we have
\[
	\mathcal{E}_p\cong\left\{ x\in \N\left[\frac{1}{p}\right]\,\sqcup\,(0,1] \mid x\leq 1 \right\}\cup\left\{ \infty \right\}
\]
where the order and sum between two elements $x,y$ are defined as in $S_p$, with the exception that $x+y=\infty$ whenever this sum is (strictly) greater than $1_c$ in $S_p$.
\end{rmk}

\begin{qst}\label{qst:elemfraisse}
It can be checked that $\mathcal{E}_p$ satisfies all the known extra axioms (O5)-(O8) that the Cuntz semigroup of any $\CatCa$-algebra satisfies. Thus, it is natural to ask: Is $\mathcal{E}_p$ the Cuntz semigroup of some $\CatCa$-algebra $A$?

Note that such a $\CatCa$-algebra $A$ would be simple, not stably finite, and not purely infinite. However, $A$ is not the $\CatCa$-algebra $B$ constructed by R{\o}rdam in \cite{R03}, since $\Cu (B)$ does not satisfy the Corona Factorization Property (see \cite[Theorem~5.8]{BP18}) but $\mathcal{E}_p$ does.
\end{qst}

\subsection{The Cantor set and the Pseudo-arc}\label{subsec:D}
Fra\"{i}ss\'{e} Theory allows one to rewrite well-known topological spaces such as the Cantor set $2^{\N}$ and the pseudo-arc $\mathbb{P}$ as Fra\"{i}ss\'{e} limits and, in particular, to generically (re)prove some interesting facts about them, such as universality and homogeneity. We refer the reader to \cite{BK22, IS06, K13} for more details. Following these results, we show that the $\Cu$-semigroups $\Lsc(2^{\N},\overline{\N})$ and $\Lsc(\mathbb{P},\overline{\N})$ are Fra\"{i}ss\'{e} limits of well-chosen categories of Cuntz semigroups. 

\begin{prg}[\textbf{Lower-semicontinuous functions}]\label{prg:LscLike}
We begin by recalling some facts about monoids of lower-semicontinuous functions, which constitute a great source of example of abstract Cuntz semigroups. For instance, it is known that the monoid of lower-semicontinuous functions from a compact, metric (or, more generally, hereditarily Lindel\"{o}f,  locally compact, and Hausdorff) space $X$ to $\overline{\N}$, denoted by $\Lsc (X, \overline{\N} )$, is a $\Cu$-semigroup when equipped with pointwise addition and order. (See \cite[Proposition~1.16]{EI22}, and also \cite[Corollary~4.22]{V22} and \cite[Theorem~5.17]{APS11}.) In the specific case where $X$ is a compact one-dimensional $\CW$-complex, then $\Lsc(X,\overline{\N})$ is in fact the concrete Cuntz semigroup of the $\CatCa$-algebra $C(X)$. (See \cite{R13,C22} and also \cite{APS11},\cite{C23a} for other examples of concrete Cuntz semigroup of $\CatCa$-algebras that can be expressed as lower-semicontinuous functions.)
 
 Further, as noted in \cite{RS10} for the interval, and in \cite{C22} for compact one-dimensional $\CW$-complexes (which include finite discrete sets, the interval and the circle), the set of $\Cu$-morphisms $\Hom_{\Cu}(\Lsc (X,\overline{\N}),T)$ can be equipped with the following \emph{$\Cu$-metric}
\[
   d_{\Cu}(\alpha,\beta):=
   \inf \left\{
r>0\mid \forall V\in\mathcal{O}(X), \alpha(1_{V})\leq\beta(1_{V_{r}})  \text{ and }  \beta(1_{V})\leq\alpha(1_{V_{r}})
   \right\} 
 \]
 where $\mathcal{O}(X):=\{$Open sets of $X\}$ and $V_r$ is an $r$-open neighborhood of $V$. 
 \end{prg}
 
 As proved implicitly in {\cite[Lemma~4.8]{V22b}} for the interval case and explicitly in \cite{C22} for compact one-dimensional $\CW$-complexes, there is a strong link between the metric $d_{\Cu}$ and finite-set comparison:
 
 \begin{prop}[{\cite[Proposition 5.6]{C22}}]
  Let $X$ be a compact one-dimensional $\CW$-complex and let ${\{\overline{U_k}\}}_1^n$ be a finite closed cover of $X$ induced by an equidistant partition of size $1/n$. 
  
Let $F_n:=\left\{f\in \Lsc(X,\overline{\N}) \mid f_{\mid U_k} \text{ is constant for any } k\in\{1,\ldots, n\}\right\}$. For any pair of morphisms $\alpha ,\beta \in \Hom_{\Cu}(\Lsc (X, \overline{\N} ),T)$, we have
  \begin{itemize}
   \item[(i)] $\alpha\simeq_{F_{n}}\beta$ implies $d_{\Cu}(\alpha ,\beta )\leq 2/n$.
   \item[(ii)] $d_{\Cu}(\alpha ,\beta )\leq 1/n$ implies $\alpha\simeq_{F_{n}}\beta$.
  \end{itemize}
 \end{prop}

As shown in \cite[Lemma 5.16]{APS11}, any continuous map $f\colon Y\longrightarrow X$ between second countable, compact, Hausdorff spaces induces a $\Cu$-morphism $\Lsc(f,\overline{\N})\colon \Lsc(X,\overline{\N})\longrightarrow \Lsc(Y,\overline{\N})$ given by  $l\longmapsto l\circ f $. In what follows, we prove that a weak converse of this result ---akin to what happens for commutative $\CatCa$-algebras--- also exists. These results might be well known to experts (for example, if $X$ and $Y$ are one-dimensional, they follow from \cite{CE08} and standard facts about abelian $\CatCa$-algebras). However, since we have not found them in the literature with our generality, we provide a proof here for the convenience of the reader.

 \begin{prop}\label{prp:DualityLsc}
  Let $X,Y$ be compact, metric spaces and let $\alpha\colon \Lsc (X,\overline{\N})\longrightarrow \Lsc (Y,\overline{\N})$ be a  $\Cu$-morphism such that $\alpha (1)=1$. 
  
Then there exists a continuous map $f_\alpha\colon Y\longrightarrow X$ inducing $\alpha$, in the sense that $\alpha=\Lsc(f_\alpha,\overline{\N})$.
 \end{prop}
 \begin{proof}
  We begin our argument by proving the following claim.
  
  \textbf{Claim.} Let $\lambda\colon \Lsc (X,\overline{\N})\longrightarrow \overline{\N}$ be a $\Cu$-morphism such that $\lambda (1)=1$. Then, there exists $x\in X$ such that $\lambda = {\rm ev}_x$, the evaluation at $x$.
  
  \emph{Proof of the Claim}. Let $J$ be the family of open subsets $U$ in $X$ such that $\lambda (\mymathbb{1}_U)=1$. Given finitely many open sets $U_1,\ldots ,U_n\in J$, we have 
  \[
   n=\lambda (\mymathbb{1}_{U_1})+\ldots + \lambda (\mymathbb{1}_{U_n}) = 
 \lambda (\mymathbb{1}_{\cup_j U_j})+\ldots + \lambda (\mymathbb{1}_{\cap_j U_j}).
  \]

This implies that $\lambda (\mymathbb{1}_{\cap_j U_j})=1$ and, in particular, that $\cap_j U_j\neq\emptyset$. Thus, $J$ has the finite intersection property and, since $X$ is compact, we get that $\cap_{U\in J}\overline{U}$ is not empty.

To see that  $\cap_{U\in J}\overline{U}=\{ x\}$, assume for the sake of contradiction that there exist $x,y\in \cap_{U\in J}\overline{U}$ with $x\neq y$. Let $B,C$ be closed disjoint balls of nonzero radius centered at $x$ and $y$ respectively. One has $1\leq \mymathbb{1}_{X-B}+\mymathbb{1}_{X-C}$ and, consequently, that $\lambda(\mymathbb{1}_{X-B})$  or $\lambda(\mymathbb{1}_{X-C})$ is $1$. Thus, we may assume without loss of generality that $\lambda(\mymathbb{1}_{X-C})=1$ or, in other words, that $X-C\in J$. This is a contradiction, since $y\notin \overline{X-C}$. It follows that $\cap_{U\in J}\overline{U}$ contains a single point $x$.

Finally, to see that $\lambda = {\rm ev}_x$, take any open subset $U$ and let $U'$ be a compactly contained open subset in $U$ such that $\lambda (\mymathbb{1}_{U})=\lambda (\mymathbb{1}_{U'})$. If $\lambda (\mymathbb{1}_{U})=1$, then $x\in \overline{U'}$ and so $x\in U$. Conversely, if $x\in U$, take $U'$ compactly contained in $U$ such that $x\in U'$. Then, $\mymathbb{1}_{U}+\mymathbb{1}_{X-\overline{U'}}\geq 1$, which implies that either $\lambda (\mymathbb{1}_{U})=1$ or $\lambda (\mymathbb{1}_{X-\overline{U'}})=1$. The second equality cannot hold, since otherwise we would get $x\in X-\overline{U'}$, a contradiction. This proves the claim.
  
  Now, for any $y\in Y$, the composition ${\rm ev}_y\circ \alpha$ is a $\Cu$-morphism such that $\lambda (1)=1$. Using the claim, there exists $x\in X$ with ${\rm ev}_y\circ \alpha = {\rm ev}_x$. Let $f_\alpha\colon Y\longrightarrow X$ be the map defined by $f_{\alpha}(y):=x$. To see that $f_{\alpha}$ is continuous, take an open subset $U$ of $X$ and let $V\subseteq Y$ be the open subset such that $\alpha (\mymathbb{1}_U)=\mymathbb{1}_V$. Then, we see that
  \[
   f_{\alpha}^{-1}(U)=\{ y\in Y \mid {\rm ev}_{f_{\alpha}(y)}(\mymathbb{1}_U)=1 \}
   = \{ y\in Y \mid {\rm ev}_y\circ \alpha(\mymathbb{1}_U)=1 \}
   = V
  \]
which ends the proof.
 \end{proof}

  \begin{cor}\label{prp:SurjDualityLsc}
  Let $X,Y$ be compact, metric spaces. Let $\alpha\colon \Lsc (X,\overline{\N})\longrightarrow \Lsc (Y,\overline{\N})$ be a $\Cu$-morphism such that $\alpha (1)=1$, and let $f_\alpha\colon Y\longrightarrow X$ be the continuous map obtained in the proposition above. 
  
  Then, $\alpha$ is an order-embedding if and only if $f_\alpha$ is surjective.
 \end{cor}
 \begin{proof}
Assume for the sake of contradiction that $\alpha$ is an order-embedding and that there exists $x\in X\setminus f_{\alpha}(Y)$. Since $f_{\alpha}(Y)$ is compact, we can find an open neighbourhood $U$ of $x$ disjoint with $f_{\alpha}(Y)$. In particular, $f_{\alpha}^{-1}(U)=\emptyset$ and, therefore, $\alpha (\mymathbb{1}_U)=\mymathbb{1}_{f_{\alpha}^{-1}(U)}=0$, a contradiction.
  
  Conversely, assume now that $f_{\alpha}\colon Y\longrightarrow X$ is surjective. Since the order on $\Lsc (X,\overline{\N})$ is determined by the indicators $\mymathbb{1}_U$ (see e.g. \cite[Proposition 4.3]{C22} or \cite{V22}), we only need to prove that $\mymathbb{1}_U\leq \mymathbb{1}_V$ whenever $\alpha (\mymathbb{1}_U)\leq \alpha (\mymathbb{1}_V)$.
Let $U,V\subseteq X$ be such that $\alpha (\mymathbb{1}_U)\leq \alpha (\mymathbb{1}_V)$. We have $\mymathbb{1}_{f_{\alpha}^{-1}(U)}\leq \mymathbb{1}_{f_{\alpha}^{-1}(V)}$. Consequently, $f_{\alpha}^{-1}(U)\subseteq f_{\alpha}^{-1}(V)$. By the surjectivity of $f_{\alpha}$ we deduce that  $U\subseteq V$, as desired.
 \end{proof}

\begin{prg}[\textbf{The Cantor Set}]
Let us recall a characterization of the Cantor set. We use the language and formulations detailed in \cite{BK22}, even though this characterization had been obtained beforehand, e.g. in \cite{S04,K13}. As mentioned in the discussion of \cite[Example 4.55]{BK22}, the Cantor set $2^{\N}$ is the Fra\"{i}ss\'{e} limit of the category of finite discrete sets and continuous surjections in the category of zero-dimensional compacta and continuous surjections. Now, using the characterization of Fra\"{i}ss\'{e} limits given in \cite[Theorem 4.15]{BK22}, one obtains the following result.
\end{prg}

\begin{thm}[Characterization of the Cantor set]
A zero-dimensional compactum $C$ is the Cantor set if and only if $C$ satisfies the following property: 

For any two finite discrete sets $F,F'$ and any two continuous surjections $f\colon C\longrightarrow F, g\colon F'\longrightarrow F$, there exists a continuous surjection $h\colon C\longrightarrow F'$ such that $h\circ g=f$.
\end{thm}

\begin{prg} \emph{The category $\mathcal{K}_{2^{\N}}$} is the category whose objects are simplicial $\Cu$-semigroups and whose morphisms are order-embeddings such that $1\mapsto 1$.
\end{prg}

\begin{thm}\label{prp:CantoerFraisse}
 The category $\mathcal{K}_{2^{\N}}$ is Fra\"{i}ss\'{e}.
\end{thm}
\begin{proof}
Throughout the proof, we will denote the $r$-tuple $(0,\ldots,1,\ldots,0)$  with value $1$ at the $i$-th component and $0$ everywhere else by $\delta^r_i$. Note that $\{\delta_i^r\}_{i=1}^r$ generates $\overline{\N}^r$.

The category $\mathcal{K}_{2^{\N}}$ contains countably many objects and finitely many morphisms between two given objects. It follows that the category $\mathcal{K}_{2^{\N}}$ is separable.

Let $ \overline{\N}^{r_1}, \overline{\N}^{r_2}$ be simplicial $\Cu$-semigroups. We construct $\alpha_1\colon \overline{\N}^{r_1}\longrightarrow \overline{\N}^{r_1}\oplus\,\,\, \overline{\N}^{r_2}$ that sends $\delta_1^{r_1}\mapsto \delta_1^{r_1}\oplus 1_{\overline{\N}^{r_2}}$ and $\delta_i^{r_1}\mapsto \delta_i^{r_1}\oplus 0_{\overline{\N}^{r_2}}$ for any $2\leq i \leq r_1$. Similarly, we construct $\alpha_2\colon \overline{\N}^{r_2}\longrightarrow \overline{\N}^{r_1}\oplus\,\,\,  \overline{\N}^{r_2}$. It is readily checked that $\alpha_1,\alpha_2$ are $\mathcal{K}_{2^{\N}}$-morphisms and, hence, the joint embedding property follows.

Let $\alpha_1\colon \overline{\N}^{r}\longrightarrow \overline{\N}^{t_1}$ and $\alpha_2\colon \overline{\N}^{r}\longrightarrow \overline{\N}^{t_2}$ be $\mathcal{K}_{2^{\N}}$-morphisms. We know from \autoref{prp:SurjDualityLsc} that $r\leq t_1,t_2$. Further, we may assume that $t_1=t_2=t$ and, upon a possible reindexing (ie. composing with an isomorphism), we may also assume that the $\alpha_i$'s are of the form $\id\oplus \eta_i$ for some $\Cu$-morphisms $\eta_i\colon  \overline{\N}^{r}\longrightarrow \overline{\N}^{t-r}$. Note that, since both $\alpha_1$ and $\alpha_2$ map $1_r$ to $1_t$, the maps $\eta_i$ also map $1_r$ to $1_{t-r}$.

Let $\beta_i\colon \overline{\N}^r\oplus \overline{\N}^{t-r}\to \overline{\N}^r\oplus \overline{\N}^{t-r}\oplus \overline{\N}^{t-r}$ be the maps $\beta_1 (x,y)= (x,y,\eta_2(x))$ and $\beta_2 (x,y)=(x,\eta_1(x),y)$. Note that these order-embeddings map $1$ to $1$. By construction, we have $\beta_1\circ \alpha_1= \beta_2\circ \alpha_2$, which shows that the category $\mathcal{K}_{2^{\N}}$ has the amalgamation property.
\end{proof}

\begin{cor}\label{prp:Cantor}
The Fra\"{i}ss\'{e} limit of $\mathcal{K}_{2^{\N}}$ is $\Lsc (2^{\N},\overline{\N})$.
\end{cor}

\begin{proof}
First, note that we can identify any $\mathcal{K}_{2^{\N}}$-object $\overline{\N}^r$ with $\Lsc(X_r,\overline{\N})$, where $X_r$ is any finite discrete set of cardinality $r$. Now, let $\alpha\colon\Lsc(X_r,\overline{\N})\longrightarrow \Lsc(X_t,\overline{\N})$ be a $\mathcal{K}_{2^{\N}}$-morphism and consider the continuous surjective map $f_\alpha\colon X_t\longrightarrow X_r$ obtained from \autoref{prp:DualityLsc}. From the construction of $f_\alpha$, we deduce that $\alpha$ can be identified with $\Lsc(f_\alpha,\overline{\N})\colon\Lsc(X_r,\overline{\N})\longrightarrow \Lsc(X_t,\overline{\N})$ which sends $l\longmapsto l\circ f_\alpha$.

On the other hand, we know that the Fra\"{i}ss\'{e} limit is obtained from an inductive system in $\mathcal{K}_{2^{\N}}$ that we write $(\overline{\N}^{r_i},\alpha_i)_i$. By the above identifications, we can identify the system with $(\Lsc(X_{r_i},\overline{\N}),\Lsc(f_{\alpha_i},\overline{\N}))_i$. Now combining \autoref{prp:SurjDualityLsc} with \autoref{prp:FraisseLimCu} (and the fact that the category $\mathcal{K}_{2^{\N}}$ has \emph{exact} amalgamation property), we get that $\lim\limits_{\longleftarrow}(X_{r_i},f_{\alpha_i})$ is a zero-dimensional compactum satisfying the above characterization of the Cantor set. That is, $\lim\limits_{\longleftarrow}(X_{r_i},f_{\alpha_i})\cong 2^{\N}$ and the result follows from \cite[Proposition 5.18]{APS11}.
\end{proof}

\begin{prg}[\textbf{The Pseudo-arc}]
Proceeding as before, we recall a characterization of the pseudo-arc in the language from \cite{BK22}, althought this result had also been obtained in the past, e.g. in \cite{IS06,K13}. As shown in \cite[Theorem 4.38]{BK22}, the pseudo-arc $\mathbb{P}$ is the Fra\"{i}ss\'{e} limit of the category consisting of a single object being the unit interval and continuous surjections in the category of arc-like continua and continuous surjections. Using \cite[Theorem 4.15]{BK22}, one gets:
\end{prg}

\begin{thm}[Characterization of the Pseudo-arc]
An arc-like continuum $P$ is the pseudo-arc if and only if $P$ satisfies the following property: 

For any two continuous surjections $f\colon P\longrightarrow [0,1]$ and $g\colon [0,1]\longrightarrow [0,1]$ and any $\varepsilon>0$, there exists a continuous surjection $h\colon P\longrightarrow [0,1]$ such that $\Vert h\circ g -f\Vert < \varepsilon$.
\end{thm}

The following two lemmas will be needed in our proofs. The first is known as the Mountain Climbing Lemma (see \cite{H52}), while the second is readily obtained by generalizing the arguments in \cite[Lemma 4.5]{V22b}.

\begin{lma}[Mountain Climbing Lemma]\label{prp:MountClim}
 Let $f_1,f_2\colon [0,1]\longrightarrow [0,1]$ be continuous, piecewise linear maps such that $f_1(0)=0=f_2(0)$ and $f_1(1)=1=f_2(1)$. Then, there exist surjective, continuous maps $g_1,g_2\colon [0,1]\longrightarrow [0,1]$ such that $f_1\circ g_1=f_2\circ g_2$.
\end{lma}

\begin{lma}\label{prp:CompLsc}
  Let $X$ and $Y$ be second countable, compact, Hausdorff spaces. Let $f,g:Y\longrightarrow X$ be continuous surjective maps and consider their induced $\Cu$-morphisms $\Lsc(f,\overline{\N}),\Lsc(g,\overline{\N})$ given by $l\longmapsto l\circ f,l\circ g$ respectively. Then,
 \begin{itemize}
\item[(i)] For any finite subset $F$ of $\Lsc (X,\overline{\N})$, there exists $\varepsilon_F >0$ such that $\Lsc(f,\overline{\N})\simeq_F \Lsc(g,\overline{\N})$ whenever $\Vert f-g\Vert<\varepsilon$.

\item[(ii)] For any $\varepsilon >0$, there exists a finite subset $F_{\varepsilon}$ of $\Lsc (X,\overline{\N})$ such that $\Vert f-g\Vert <\varepsilon$ whenever $\Lsc(f,\overline{\N})\simeq_F \Lsc(f,\overline{\N})$.
\end{itemize}

Further, the $\varepsilon_F$ only depends on $F,X,Y$, and  $F_{\varepsilon}$ only depends on $\varepsilon ,X,Y$ (not on $f,g$).
 \end{lma}

\begin{prg} \emph{The category $\mathcal{K}_\mathbb{P}$} is the category whose (single) object is $\Lsc ([0,1],\overline{\N})$ and whose morphisms are order-embeddings such that $1\mapsto 1$.
\end{prg}

\begin{thm}\label{prp:BingFraisse}
 The category $\mathcal{K}_\mathbb{P}$ is Fra\"{i}ss\'{e}.
\end{thm}
\begin{proof}
  The category $\mathcal{K}_\mathbb{P}$ contains only one object, so the joint embedding property is trivial. 

 To see that  $\mathcal{K}_\mathbb{P}$ is separable, we consider the subcategory $\mathfrak{s}\subseteq \mathcal{K}_\mathbb{P}$ whose morphisms are $\mathcal{K}_\mathbb{P}$-morphisms of the form $\Lsc(h,\overline{\N})\colon \Lsc([0,1],\overline{\N})\longrightarrow \Lsc([0,1],\overline{\N})$ where $h\colon [0,1]\longrightarrow [0,1]$ is any piecewise linear, surjective map with rational, finitely many peaks and valleys. We know from \autoref{prp:SurjDualityLsc} that any $\mathcal{K}_\mathbb{P}$-morphism $\alpha$ is of the form $\Lsc(f_\alpha,\overline{\N})$ where $f_\alpha\colon [0,1]\longrightarrow [0,1]$ is a continuous surjective map. Moreover, it is well-known that any such continuous surjective map can be approximated in norm by a piecewise linear map $h$ with rational, finitely many peaks and valleys. Therefore, it follows from \autoref{prp:CompLsc} that $\mathfrak{s}$ is a countable, dominating subcategory.

Finally, let us prove $\mathcal{K}_\mathbb{P}$ satisfies the near amalgamation property. Let $\alpha_1,\alpha_2$ be $\Cu$-morphisms in $\mathcal{K}_\mathbb{P}$ and consider the continuous, surjective maps $f_{\alpha_1},f_{\alpha_2}\colon [0,1]\longrightarrow [0,1]$ obtained from \autoref{prp:SurjDualityLsc}. Also, let $F\subseteq \Lsc ([0,1],\overline{\N})$ be a finite subset and consider the bound $\varepsilon_F>0$ given by \autoref{prp:CompLsc}.

Without loss of generality, we can assume that $f_{\alpha_1}(0)=0=f_{\alpha_2}(0)$ and $f_{\alpha_1}(1)=1=f_{\alpha_2}(1)$. (If needed, we can precompose them with well-chosen continuous surjective maps.) As before, we can find piecewise linear maps $h_1,h_2\colon [0,1]\longrightarrow [0,1]$ with rational, finitely many peaks and valleys, at distance less than $\varepsilon /2$ from $f_{\alpha_1}$ and $f_{\alpha_2}$ respectively, and such that $h_1(0)=0=h_2(0)$ and $h_1(1)=1=h_2(1)$.

Applying the Mountain Climbing lemma (see \autoref{prp:MountClim}), we obtain surjective continuous maps $g_1,g_2 \colon [0,1]\longrightarrow [0,1]$ such that $h_1\circ g_1 = h_2\circ g_2$. By construction, we get that $f_{\alpha_1}\circ g_1$ and $f_{\alpha_2}\circ g_2$ are at distance at most $\varepsilon$. Let $\beta_1:=\Lsc(g_1,\overline{\N})$ and $\beta_2:=\Lsc(g_2,\overline{\N})$ be the $\mathcal{K}_\mathbb{P}$-morphisms induced by $g_1$ and $g_2$ respectively. By \autoref{prp:CompLsc}, we obtain $\beta_1\circ\alpha_1\simeq_F \beta_2\circ\alpha_2$, as desired.
\end{proof}

\begin{cor}\label{prp:PseudoArc}
The Fra\"{i}ss\'{e} limit of $\mathcal{K}_\mathbb{P}$ is $\Lsc (\mathbb{P},\overline{\N})$.
\end{cor}
\begin{proof}
We know that the Fra\"{i}ss\'{e} limit is obtained from an inductive system in $\mathcal{K}_\mathbb{P}$ that we write $(\Lsc([0,1],\overline{\N}),\alpha_i)_i$. Now combining \autoref{prp:SurjDualityLsc} with \autoref{prp:FraisseLimCu}, we get that $\lim\limits_{\longleftarrow}([0,1],f_{\alpha_i})$ is an arc-like continuum satisfying the above characterization of the pseudo-arc. That is, $\lim\limits_{\longleftarrow}([0,1],f_{\alpha_i})\cong \mathbb{P}$ and the result follows from \cite[Proposition 5.18]{APS11}.
\end{proof}

\begin{rmk}\label{rmk:AltPrfCantPseud}
We have chosen to give a self-contained proofs of the last two examples, but they could have alternatively been obtained as a combination of \autoref{rmk:endingpartc}~(ii), \autoref{prp:SurjDualityLsc}, \cite{R13} and known results in Fra\"{i}ss\'{e} theory of $\CatCa$-algebras (for example, those in \cite{Vig22}). Indeed, since in both categories $\mathcal{K}_{2^{\N}}$ and $\mathcal{K}_\mathbb{P}$ the dimension of the underlying spaces is at most one, we know from \cite{R13} (see also \cite{C07} for the case of the interval) that all the objects are Cuntz semigroups of commutative $\CatCa$-algebras. Further, it follows from \autoref{prp:SurjDualityLsc} that the $\Cu$-morphisms in both categories are in correspondence with surjective continuous maps between the spaces, which in turn correspond to injective $^*$-homomorphisms between the commutative $\CatCa$-algebras.

In other words, this shows that both $\mathcal{K}_{2^{\N}}$ and $\mathcal{K}_\mathbb{P}$ can be written as $\Cu (\cc_{2^{\N}})$ and $\Cu (\cc_\mathbb{P})$ for well chosen subcategories of commutative $\CatCa$-algebras, which are well-known to be Fra\"{i}ss\'{e} (e.g. \cite[Theorem~3.4]{Vig22} for the case of the pseudo-arc). Using \autoref{rmk:endingpartc}~(ii), it follows that both $\mathcal{K}_{2^{\N}}$ and $\mathcal{K}_\mathbb{P}$ are Fra\"{i}ss\'{e}, and that their limit coincides with the Cuntz semigroup of the limit of $\cc_{2^{\N}}$ and $\cc_\mathbb{P}$.
\end{rmk}

\subsection{The Jiang-Su algebra}\label{subsec:E}

It is readily checked that all the results developed in this paper also hold for classes where the objects are pairs of the form $(S,\rho)$ with $\rho\colon S\to [0,\infty ]$ a generalized $\Cu$-morphism (also known as a \emph{functional}) and where the morphisms from a pair $(S,\rho)$ to $(T,\delta)$ are simply $\Cu$-morphisms $S\to T$ preserving the prescribed functional $\rho$. Namely, the statements still hold because a limit of functional-preserving $\Cu$-morphisms is still functional-preserving.

Using \autoref{thm:cstarcu} (and \cite{R12}) in conjunction with the fact that quasitraces on a $\CatCa$-algebra correspond to functionals on its Cuntz semigroup (\cite{ERS11}), one obtains the following result by using \cite{M17} (see also \cite{EFHKKM16}).

\begin{prg} \emph{The category $\mathcal{K}_Z$} is the category whose objects are pairs $(Z_{p,q},\rho )$ with $p,q$ coprime and $\rho$ a faithful functional, and whose morphisms are functional-preserving order-embeddings. Here, $Z_{p,q}$ denotes the Cuntz semigroup of the prime dimension-drop algebra $\mathcal{Z}_{p,q}$. Using computations in \cite{APS11}, recall that we have
\[
	\Cu (\mathcal{Z}_{p,q})\cong Z_{p,q}=\{ f\in\Lsc ([0,1],\overline{\N}) \mid \exists k_1,k_2 \text{ such that }f(0)=qk_1\text{ and }f(1)=pk_2 \}.
\]
\end{prg}

\begin{thm}
The class $\mathcal{K}_Z$ is Fra\"{i}ss\'{e} and its limit is $Z$, the Cuntz semigroup of the Jiang-Su algebra $\mathcal{Z}$.
\end{thm}

\section{\texorpdfstring{Metrics on ${\rm Hom}_{\Cu}$-sets}{Metrics on HomCu-sets}}\label{sec:CuMetrics}

In this last section we introduce a metric on any ${\rm Hom}_{\Cu}$-set building on the ideas of \cite{C22, CES11, RS10, V22b}. We also provide a number of examples and, in particular, we show that this notion generalizes the metrics introduced in \cite[Definition~5.1]{C22}. We also study the relation between our proposed metric and finite-set comparison. Further, we prove prove that, in general, comparing $\Cu$-morphisms via the metric is more restrictive than using finite-set comparison. As a consequence, when (re)formulating the notion of Cauchy sequences in terms of the metric, the limit we obtain might not behave as expected. 

Let us start by recalling an important fact about the category $\Cu$.
\begin{prg}[\textbf{A generator for $\Cu$}] 
 Let $\mathbb{G}$ be the submonoid of $\Lsc ([0,1], \overline{\N} )$ defined as
 \[
\mathbb{G} = 
  \{f\in \Lsc([0,1],\overline{\N})\mid 
  f(0)=0,\, f
  \text{ increasing}
  \}
  .
 \]
If follows from \cite[Section~5.2]{S18} that $\mathbb{G}$ is a $\Cu$-semigroup. Moreover, $\mathbb{G}$ is a generator for the category $\Cu$, in the sense that the functor $\Cu(\mathbb{G},-)\colon \Cu\longrightarrow \rm Set$ is faithful. 

Using terminology of \cite{C22}, we can view $\mathbb{G}$ as the sub-$\Cu$-semigroup of $\Lsc ([0,1], \overline{\N})$ chain-generated by $\Lambda_{\mathbb{G}}:=\{\mymathbb{1}_{(t,1]}\mid t\in (0,1]\}$. Therefore, the order and the compact-containment relation in $\mathbb{G}$ are completely determined by the ones in $\Lambda_{\mathbb{G}}$. (See \cite[Section 4]{C22} for more details.)
\end{prg}


\begin{prg}[\textbf{Thomsen semigroup of a $\Cu$-semigroup}]\label{prg:thmsencu}
Let $S$ be a $\Cu$-semigroup. We define the \emph{Thomsen semigroup of $S$}, in symbols $\Th(S)$, to be the set of $\Cu$-morphisms from the generator $\mathbb{G}$ to $S$. In other words, 
\[
\Th(S):=\rm Hom_{\Cu}(\mathbb{G},S).
\]
\end{prg}

This construction is inspired by the $\CatCa$-case, where the \emph{Thomsen semigroup of a $\CatCa$-algebra $A$}, in symbols $\mathcal{T}h(A)$, is the set of approximate unitary equivalence classes of $^*$-homomorphisms of the form $C_0((0,1])\longrightarrow A\otimes\mathcal{K}$. (See \cite{Th92}.) Here, note that $C_0((0,1])$ is a generator for the category $\CatCa$. (See e.g. \cite{S18}.)

Therefore, the construction above is the natural $\Cu$-analogue of the Thomsen semigroup for $\CatCa$-algebras. In fact, there exists a natural (monoid) map $\mathcal{T}h (A)\longrightarrow \Th (\Cu (A))$, defined in the following proposition.

\begin{prop}
Let $A$ be a $\CatCa$-algebra. The map $\iota:\mathcal{T}h (A)\longrightarrow \Th (\Cu (A))$ given by $[\varphi ]\mapsto  \Cu (\varphi)_{\mid\mathbb{G}}$ is a well-defined monoid morphism.

If $A$ is stable and has stable rank one, then $\iota$ is a bijection.
\end{prop}

\begin{proof}
It is a well-known fact that any two approximate unitary equivalent $^*$-homomorphisms agree at level of $\Cu$. Therefore, $\iota$ is a well-defined map for any $\CatCa$-algebra.

Now assume that $A$ has stable rank one. Then, it follows from \cite[Theorem~4.3,~Lemma~7.2]{RW10} that $\Cu(A)$ is weakly cancellative and satisfies (O5). 

Let $\tau\in\Th (\Cu (A))$. Proceeding as in the proof of \cite[Proposition~3.4]{V22b}, we can construct a $\Cu$-morphism $\tilde{\tau}\colon \Cu (C_0 ((0,1]))\longrightarrow \Cu (A)$ extending $\tau$. Further, such a morphism is unique.

The functor $\Cu$ classifies $^*$-homomorphisms from $C_0((0,1])$ to any $\CatCa$-algebra of stable rank one (See \cite{CE08, RS10, R12}). In particular, one can lift the extension $\tilde{\tau}$ to a $^*$-homomorphism $C_0((0,1])\to A$. This proves that $\iota$ is surjective.

Since the extension $\tilde{\tau}$ is unique, and the lift of any such $\tau$ is unique up to approximate unitary equivalence, $\iota$ is injective.
\end{proof}

We will use the Thomsen semigroup of $S$ to build a metric on any $\rm Hom_{\Cu}$-set $\Hom_{\Cu}(S,T)$. Let us start by equipping $\Th(S)$ with the following metric, modeled after the distance in \cite{CE08}. (See also \cite{RS10} and \autoref{prg:LscLike}.)

\begin{dfn}\label{dfn:DistG}
Let $S$ be a $\Cu$-semigroup, and let $\alpha,\beta\in \Th(S)$. We define 
\[
d_{\mathbb{G}}(\alpha,\beta):=
\inf \left\{
r>0\mid \forall t\in [0,1], \, \alpha\left(\mymathbb{1}_{(t+r,1]}\right)
\leq\beta\left(\mymathbb{1}_{(t,1]}\right) 
\text{ and }
\beta\left(\mymathbb{1}_{(t+r,1]}\right)
\leq \alpha\left(\mymathbb{1}_{(t,1]}\right)
\right\} .
\]
\end{dfn}

\begin{rmk}\label{rmk:DistGMetric}
Note that, by definition, $d_{\mathbb{G}}(\alpha,\beta)=0$ precisely when $\alpha(\mymathbb{1}_{(t,1]})=\beta(\mymathbb{1}_{(t,1]})$ for every $t$. Thus, since $\mathbb{G}$ is generated (as a $\Cu$-semigroup) by the elements $\mymathbb{1}_{(t,1]}$, one gets $\alpha = \beta$. Consequently, $d_{\mathbb{G}}$ is a metric. This is in contrast to \cite{CE08}, where weak cancellation of the Cuntz semigroup is needed to ensure positivity. 
\end{rmk}

\begin{prg}[\textbf{Sets with generating image}]
Let $\Lambda$ be a subset of $\Th (S)$. We say that $\Lambda$ has a \emph{generating image in $S$} if the submonoid generated by $\{ \tau (\mymathbb{1}_{(t,1]}) \mid t\in [0,1],\, \tau\in\Lambda \}$ is sup-dense in $S$. 

Equivalently, $\Lambda$ has a generating image if for any $s',s\in S$ with $s'\ll s$, there exist $\tau_1,\ldots ,\tau_n$ in $\Lambda$ and $t_1,\ldots , t_n\in [0,1]$ such that 
$s'\ll \tau_1 (\mymathbb{1}_{(t_1,1]}) +\ldots + \tau_n (\mymathbb{1}_{(t_n,1]}) \ll s$.
\end{prg}

For any $\Cu$-semigroup $S$, there always exists a family (and, in fact, many) with a generating image. For instance, the following result shows that $\Th(S)$ always has a generating image in $S$.

\begin{prop}[{\cite[Lemma~5.16]{S18}}]\label{prp:IniMap}
Let $S$ be a $\Cu$-semigroup and let $(x_n)_n$ be a $\ll$-increasing sequence in $S$. Then, there exists $\tau \in\Th (S)$ such that $\tau (\mymathbb{1}_{(\frac{1}{n},1]}) =x_n$.
\end{prop}

\begin{proof}
We give a brief argument for the convenience of the reader. 

Using \cite[Proposition~2.10]{APT20}, there exists a net $(y_t)_{t\in [0,1)}$, with $y_{\frac{1}{n}}=x_n$, such that $y_t\ll y_r$ whenever $r<t$ and $\sup_{t>r} y_t = y_r$. (This is achieved by an iterated application of (O2).) We let $\tau\colon\mathbb{G}\longrightarrow S$ be the $\Cu$-morphism defined by $\tau (\mymathbb{1}_{(t,1]} ):=y_t$ for each $t\in [0,1)$.
\end{proof}

\begin{dfn}\label{dfn:distCu}
Let $S,T\in\Cu$ and let $\Lambda\subseteq \Th (S)$ be a subset with a generating image in $S$. For any $\alpha,\beta\in \Hom_{\Cu}(S,T)$, we define
\[
d_{\Lambda}(\alpha,\beta):=\sup_{\tau\in \Lambda}d_\mathbb{G}(\alpha\circ\tau,\beta\circ\tau).
\]
\end{dfn}

\begin{lma}\label{prp:dCuMetric}
 Let $S,T\in\Cu$ and let $\Lambda\subseteq \Th (S)$ with a generating image in $S$. Then $d_{\Lambda}(\alpha,\beta)$ is a metric on $\Hom_{\Cu}(S,T)$.
\end{lma}
\begin{proof}
The symmetry, triangular inequality, and the fact that $d_{\Lambda}(\alpha,\beta)=0$ whenever $\alpha=\beta$ are all immediate. We are left to show that $\alpha=\beta$ whenever $d_{\Lambda}(\alpha,\beta)=0$.

Thus, assume that $d_{\Lambda}(\alpha,\beta)=0$. Let $s',s\in S$ be such that $s'\ll s$. Since $\Lambda$ has a generating image in $S$, there exist $\tau_1,\ldots ,\tau_n\in \Th (S)$ and $t_1,\ldots , t_n\in [0,1]$ such that $s'\leq \tau_1 (\mymathbb{1}_{(t_1,1]}) +\ldots + \tau_n (\mymathbb{1}_{(t_n,1]}) \leq s$. Note that $d_\mathbb{G}(\alpha\circ\tau_i,\beta\circ\tau_i)=0$ for any $i\leq n$. Therefore, we deduce that $\alpha\circ\tau_i = \beta\circ\tau_i$ for any $i\leq n$. Consequently, we have 
\[
\beta (s')\leq \beta (\tau_1 (\mymathbb{1}_{(t_1,1]})) +\ldots + \beta (\tau_1 (\mymathbb{1}_{(t_n,1]}))
= \alpha (\tau_1 (\mymathbb{1}_{(t_1,1]})) +\ldots + \alpha( \tau_1 (\mymathbb{1}_{(t_n,1]}))\leq 
\alpha (s).
\]
A symmetric argument gives us that $\alpha (s')\leq \beta (s)$. Finally, a standard argument using (O2) implies $\alpha=\beta$. (See \autoref{rmk:CharacEqFiSetComp}.)
\end{proof}

\begin{exa}\label{exa:distTriv}
Let $S,T$ be $\Cu$-semigroups. The metric $d_{\Th (S)}$ is trivial, in the sense that $d_{\Th (S)}(\alpha,\beta)=1$ if and only if $\alpha\neq \beta$. We give a brief argument below.
 
Let $\alpha ,\beta \in \Hom_{\Cu} (S,T)$ be such that $d_{\Th (S)}(\alpha ,\beta )<1$ and let $\varepsilon>0$ be such that $\varepsilon>d_{\Th (S)}(\alpha ,\beta)$. Now take any pair $x',x\in S$ such that $x'\ll x$. By \autoref{prp:IniMap}, we can find $\tau\in\Th (S)$ satisfying
 \[
  \tau (\mymathbb{1}_{(0,1]})=x\quad\text{and}\quad 
  \tau (\mymathbb{1}_{(\varepsilon ,1]})=x'.
 \]
Since $d_{\mathbb{G}}(\alpha\circ\tau ,\beta\circ\tau )<\varepsilon$, we compute
 \[
  \alpha (x') = 
  \alpha (\tau (\mymathbb{1}_{(\varepsilon ,1]}))\leq 
  \beta (\tau (\mymathbb{1}_{(0,1]})) = 
  \beta (x)\quad\text{and}\quad 
  \beta (x')\leq \alpha (x).
 \]
Since this can be done for any pair $x',x$, a standard argument using (O2) gives us again that $\alpha=\beta$.
\end{exa}

The example above illustrates that, despite always having many sets with a generating image, they will only induce meaningful metrics as long as they are not too large. The following examples show that all the (meaningful) $\Cu$-metrics considered in past for specific $\rm Hom_{\Cu}$-sets can be recovered as $d_\Lambda$ from well-chosen $\Lambda$'s with generating image.

\begin{exa}\label{exa:distW}
(i)  The Cuntz semigroup of the Jacelon-Razak algebra $\mathcal{W}$ can be identified with $[0,\infty ]$; see \cite{R13b}. Let $\tau\colon \mathbb{G}\longrightarrow [0,\infty ]$ be the $\Cu$-morphism determined by $\mymathbb{1}_{(t,1]}\mapsto 1-t$. 

The family $\Lambda = \{ \tau \}$ has a generating image in $\Cu(\mathcal{W})$. 
It can be computed that, for any $\alpha ,\beta \in\Hom_{\Cu} (\Cu (\mathcal{W}),T)$, we have
 \[
  d_\Lambda (\alpha ,\beta ) = d_{\mathbb{G}}(\alpha\circ\tau ,\beta\circ\tau )
 = \inf \{ r>0 \mid \forall t\in [0,1],\, \alpha (t-r) \leq \beta (t)\text{ and }\beta(t-r)\leq \alpha (t)\}.
 \]
 
 \noindent (ii) The Cuntz semigroup of the Jiang-Su algebra $\mathcal{Z}$ can be identified with $(0,\infty ]\sqcup\, \mathbb{N}$; see e.g \cite[Theorem 7.3]{GP22}. Similarly, let $\tau\colon \mathbb{G}\longrightarrow [0,\infty ]$ be the $\Cu$-morphism determined by $\mymathbb{1}_{(t,1]}\mapsto 1-t$ and let $c\colon \mymathbb{1}_{(t,1]}\mapsto 1_c$ be the constant map. 
 
 The family $\Lambda = \{ \tau, c \}$ has a generating image in $\Cu(\mathcal{Z})$. 
One can show that, for any $\alpha ,\beta \in\Hom_{\Cu} (\Cu (\mathcal{Z}),T)$, we have
 \[
  d_\Lambda (\alpha ,\beta ) =
  \left\{  
 \begin{array}{ll}
 1 \text{, whenever } \alpha(1_c)\neq \beta(1_c).\\
 \inf \{ r>0 \mid \forall t\in [0,1],\,\alpha (t-r) \leq \beta (t)\text{ and }\beta(t-r)\leq \alpha (t)\} \text{, otherwise}.
 \end{array}
 \right.
 \]
\end{exa}

\begin{exa}\label{exa:distCa}
 Let $A$ be a $\CatCa$-algebra and let $x\in\Cu (A)$. Fix a contraction $a_x\in (A\otimes\mathcal{K})_+$ such that $x=[a_x]$. This element gives rise to the $^*$-homomorphism $\varphi_x\colon C_0 ((0,1])\longrightarrow (A\otimes\mathcal{K})_+$ given by $\id_{C_0(0,1]}\mapsto a_x$. Denote the canonical inclusion from $\mathbb{G}$ to $\Cu (C_0 (0,1]))$ by $\iota$. 
 
  The family $\Lambda = \{ \Cu (\varphi_x)\circ\iota \mid x\in \Cu(A)\}$ has a generating image in $\Cu(A)$. (In fact, $\Lambda(\mathbb{G}) = \Cu (A)$). For any $\phi_1 ,\phi_2 \in\Hom_{\CatCa} (A,B)$, one gets
 \[
  d_\Lambda (\Cu (\phi_1),\Cu (\phi_2) ) \leq 
  d_{\CatCa} (\phi_1 ,\phi_2 ).
 \]
\end{exa}

\begin{exa}\label{exa:distG}
Let $X$ be a compact, metric space. Note that, for any $x\in X$, we have
\[
\bigcap\limits_{t\in (0,1]} B(x, t \diam(X))=\{ x\} \quad\text{and}\quad \bigcup\limits_{t\in (0,1]} B(x, t \diam(X))=X.
\]
Let $S:=\Lsc (X,\overline{\N})$ and let $x\in X$. We define $x_0:=0_S$ and $x_{t}:=\mymathbb{1}_{B(x, t\diam(X))}$ for any $t\in (0,1]$.

Now, from \autoref{prp:IniMap}, we know that there exists a $\Cu$-morphism $\tau_x\colon \mathbb{G}\longrightarrow S$ such that $\tau_x (\mymathbb{1}_{(t,1]})=x_{1-t}$ for any $t\in [0,1]$.

The family $\Lambda = \{ \tau_x \mid x\in X\}$ has a generating image in $S$. 
It can be computed that for any $\alpha ,\beta \in\Hom_{\Cu} (S,T)$ their distance $d_\Lambda$ is
\[
d_\Lambda (\alpha, \beta ) =  \inf \left\{
r>0\mid \forall V\in\mathcal{O}(X), \alpha(1_{V})\leq\beta(1_{V_{r}})  \text{ and }  \beta(1_{V})\leq\alpha(1_{V_{r}})
\right\}
\]
where $\mathcal{O}(X):=\{$Open sets of $X\}$ and $V_r$ is an $r$-open neighborhood of $V$. 
In particular, this generalizes the distance  recalled in \autoref{prg:LscLike} when $X$ is a compact, one-dimensional $\CW$-complex.
\end{exa}

\begin{exa}\label{exa:UniBasis} We refer the reader to \cite[Section 5.1]{C22} for details on finite uniform bases and induced $\Cu$-semimetrics.

Let $S$ be a uniformly-based $\Cu$-semigroup with a finite uniform basis $\Lambda_f=(\Lambda_n,\varepsilon_n)_n$. Let $n\in\N$. Recall that $\Lambda_n$ is finite and, in particular, that it has finitely many \emph{chains}, i.e. finite $\ll$-increasing sequences. Let us denote the set of chains in $\Lambda_n$ starting at $0_S$ by $\mathrm{C}_n$. 

Now let $c\in \mathrm{C}_n$ and let $l_c$ be the cardinal of $c$. From \autoref{prp:IniMap}, we know that there exists a $\Cu$-morphism $\tau_c\colon \mathbb{G}\longrightarrow S$ such that $\tau_c (\mymathbb{1}_{((l_c-k)/l_c,1]})=c(k)$ for any $0\leq k\leq l_c$.

The family $\Lambda=\bigcup\limits_{n\in\N}\{ \tau_c\mid c\in \mathrm{C}_n\}$ has a generating image in $S$. 
Following the ideas of \cite[Proposition 5.5]{C22}, one can show that the metric $d_{\Lambda}$ is topologically equivalent to the $\Cu$-semimetric $dd_{\Cu,\Lambda_f}$ induced by the finite uniform basis $\Lambda_f$.
\end{exa}

The following lemma states explicitly the relation between the metrics we have built and finite-set comparison for $\Cu$-morphisms.

 \begin{lma}\label{prp:DistFinComp}
Let $S,T\in\Cu$ and let $\Lambda\subseteq \Th (S)$ with a generating image in $S$. Then
  \begin{itemize}
   \item[(i)] For any finite set $F\subseteq S$, there exists $\varepsilon_F>0$ such that $\alpha\simeq_{F}\beta$ whenever $d_{\Lambda}(\alpha ,\beta )<\varepsilon_F$.
  \end{itemize}
If moreover $\Lambda$ is finite, then
\begin{itemize}
\item[(ii)] For any $\varepsilon>0$, there exists a finite set $F_\varepsilon\subseteq S$ such that $d_{\Lambda}(\alpha ,\beta )<\varepsilon$ whenever $\alpha\simeq_{F_\varepsilon}\beta$.
\end{itemize}
 \end{lma}
\begin{proof}
(i) Note that it is enough to prove the result for  $F=\{ x',x\}$ where $x',x\in S$ are such that $x'\ll x$.
Thus, let $x',x\in S$ be such that $x'\ll x$. We know that there exist $n\in \N$, $\tau_1,\ldots ,\tau_n\in\Lambda$, and $t_1,\ldots ,t_n\in [0,1]$ such that $x'\leq \tau_1 (\mymathbb{1}_{(t_1,1]}) +\ldots + \tau_n (\mymathbb{1}_{(t_n,1]}) \ll x$.

For each $i$, let $r_i>0$ be such that $x'\ll \tau_1 (\mymathbb{1}_{(t_1+r_1,1]}) +\ldots + \tau_n (\mymathbb{1}_{(t_n+r_n,1]})$. Set $\varepsilon_F :=\min_i r_i$, and let $\alpha ,\beta$ be such that  $d_{\Lambda}(\alpha ,\beta )<\varepsilon_F$. We get
\[
 \alpha (x')\leq \alpha (\tau_1 (\mymathbb{1}_{(t_1+r_1,1]})) +\ldots+ \alpha (\tau_n (\mymathbb{1}_{(t_n+r_n,1]}))\leq 
\beta  (\tau_1 (\mymathbb{1}_{(t_1,1]})) +\ldots + \beta  ( \tau_n (\mymathbb{1}_{(t_n,1]}))\leq \beta (x)
\]
A symmetric argument gives $\beta (x')\leq \alpha (x)$, as required.

(ii) We now assume that $\Lambda$ is finite. Let $\varepsilon >0$ and let $t_1,\ldots ,t_n$ be a partition of $[0,1]$ such that $\vert t_i-t_{i+1}\vert < \varepsilon /2$. Let us define
\[
F_\varepsilon:=\{ \tau (\mymathbb{1}_{(t_i,1]}) \mid \tau\in\Lambda,\, i\leq n \}\cup\{  \tau (\mymathbb{1}_{(t_i+\varepsilon/2 ,1]}) \mid \tau\in\Lambda,\, i\leq n \}.
\]

Let $\alpha ,\beta$ be such that $\alpha\simeq_{F_\varepsilon}\beta$. By \cite[Lemma~4.8]{V22b}, we have $d_{\mathbb{G}}(\alpha\circ\tau ,\beta\circ\tau )<1/n + \varepsilon /2 \leq \varepsilon$ for every $\tau\in\Lambda$. This implies $d_\Lambda (\alpha ,\beta )<\varepsilon $, as desired.
\end{proof}

Let $S,T$ be $\Cu$-semigroups and let $\Lambda\subseteq \Th(S)$ be a set with a generating image in $S$. We will say that a sequence $(\alpha_i)_i$ in $\Hom_{\Cu}(S,T)$ is \emph{$d_\Lambda$-Cauchy} if $\sum_i d_\Lambda (\alpha_i, \alpha_{i+1})<\infty$. The proposition below shows that any $d_\Lambda$-Cauchy sequence has a unique limit $\alpha$, in the sense of \autoref{dfn:LimitLin}. Nevertheless, an extra-assumption is needed (e.g. $\Lambda$ is finite) in order for this limit $\alpha$ to satisfy $d_\Lambda (\alpha_i , \alpha)\rightarrow 0$.

\begin{prop}\label{prp:LimMorph}
Let $S,T$ be $\Cu$-semigroups and let $\Lambda\subseteq \Th(S)$ be a set with a generating image in $S$. Then any $d_\Lambda$-Cauchy sequence  $(\alpha_i)_i$ in $\Hom_{\Cu}(S,T)$ has a (unique) limit. 
\end{prop}

\begin{proof}
Let $F\subseteq S$ be a finite set and let $\varepsilon_F$ be the bound given by \autoref{prp:DistFinComp}~(i). Since $(\alpha_i)_i$ is $d_\Lambda$-Cauchy, there exists some $i_F\in\N$ such that $d_\Lambda (\alpha_j ,\alpha_k)<\varepsilon_F$ whenever $j,k\geq i_F$. It follows from \autoref{prp:DistFinComp}~(i) that $(\alpha_i)_i$ is Cauchy in the sense of \autoref{dfn:CauchySeq}. 

Using \autoref{thm:UniLimComp}, one gets that  $(\alpha_i)_i$ has a (unique) limit.
\end{proof}

\begin{prop}\label{rmk:LambdaFin}
Let $S,T\in\Cu$ and let $\Lambda\subseteq \Th (S)$ with a generating image in $S$. Let $(\alpha_i)_i$ be a $d_\Lambda$-Cauchy sequence in $\Hom_{\Cu}(S,T)$ and let $\alpha\in\Hom_{\Cu}(S,T)$. Then the following are equivalent:
\begin{itemize}
\item[(i)] $\alpha$ is the limit of a sequence $(\alpha_i)_i$.
\item[(ii)] $d_{\mathbb{G}}(\alpha_i \circ\tau, \alpha\circ\tau)\rightarrow 0$ for any $\tau\in\Lambda$. 
\end{itemize}
If moreover $\Lambda$ is finite, then (i)-(ii) are in turn equivalent to
\begin{itemize}
\item[(iii)]$d_{\Lambda}(\alpha, \alpha_i)\rightarrow 0$.
\end{itemize}
\end{prop}

\begin{rmk}
Condition (iii) in the proposition above is stronger than (ii), in the sense that (iii) implies (ii). Let us exhibit an example where the converse does not hold. 

For any $n\in\N$ we let $\tau_n\in \Th (\mathbb{G})$ be the map given by $\mymathbb{1}_{(t,1]}\mapsto \mymathbb{1}_{(t+1/n,1]}$. It can be checked that $\id_{\mathbb{G}}$ is the limit of the sequence $(\tau_n)_n$.

Consider the piecewise linear functions $f_n \colon [0,1]\rightarrow [0,1]$ mapping $0\mapsto 0$, $1/2\mapsto 1/n$, and $1\mapsto 1$. Let $\lambda_n\in\Th (S)$ be the $\Cu$-morphisms defined by $\lambda_n (\mymathbb{1}_{(t,1]}):=\mymathbb{1}_{(f_n(t),1]}$. 

The family $\Lambda = \{\lambda_n \}_n$ clearly has a  generating image in $S$, and one gets $d_\mathbb{G} (\tau_n \circ \lambda_n, \lambda_n)=1/2$ for each $n$. This shows that $d_\Lambda (\tau_n,\id)$ is constantly $1/2$. In particular, the distance does not tend to $0$.
\end{rmk}

\begin{qst}
 Let $S$ be a $\Cu$-semigroup. When does there exist $\Lambda\subseteq \Th (S)$ such that $\sum_i d_\Lambda (\alpha_i, \alpha_{i+1})<\infty$ implies $d_\Lambda (\alpha ,\alpha_i)\rightarrow 0$?
\end{qst}


\end{document}